\documentclass[reqno]{amsart}

\usepackage{amsfonts}
\usepackage{amsmath,amssymb}
\usepackage{amscd}
\usepackage{amsthm}
\usepackage{fancyhdr}
\usepackage{color}
\usepackage{hyperref}
\usepackage{tikz-cd}
\usepackage[all,poly,web,knot]{xy}
\usepackage[shortlabels]{enumitem}

\theoremstyle{plain}
\newtheorem{lem}{Lemma}[subsection]
\newtheorem{thm}{Theorem}[subsection]

\newtheorem{cor}{Corollary}[subsection]
\newtheorem{prop}{Proposition}[subsection]

\theoremstyle{definition}
\newtheorem{defi}{Definition}[subsection]
\newtheorem{ex}{Example}[subsection]
\theoremstyle{remark}

\newtheorem*{rem}{Remark}

\newcommand{\RR}{\mathbb{R}}

\newcommand{\PHB}{\mbox{PHB}}
\newcommand{\HolBis}{\CH(\CF)^{AS}}

\newcommand{\CF}{\mathcal{F}}

\newcommand{\CH}{\mathcal{H}}

\newcommand{\CP}{\mathcal{P}}

\newcommand{\into}{\hookrightarrow}

\newcommand{\til}{\widetilde}

\newcommand{\CI}{C^\infty}

\newcommand{\SR}{\underline{\mathbb{R}}}
\newcommand{\SX}{\mathfrak{X}}

\newcommand{\de}{\partial}
\newcommand{\fto}{\rightarrow}

\newcommand{\bt}{\mathbf{t}}
\newcommand{\bs}{\mathbf{s}}

\newcommand{\Cinf}{{C^\infty_M}}

\newcommand{\s}{\mathbf{s}}
\renewcommand{\t}{\mathbf{t}}
\DeclareMathOperator{\Diff}{Diff}
\DeclareMathOperator{\Hol}{Hol}
\DeclareMathOperator{\Mon}{\Pi_1}
\DeclareMathOperator{\ev}{ev}

\newcommand{\HT}{\mathrm{HT}}

\renewcommand{\d}[1]{\ensuremath{\operatorname{d}\!{#1}}}

\title{Integration of singular foliations via paths}

\begin{document}

\title{Integration of singular foliations via paths}

\author{Alfonso Garmendia, Universität Potsdam and Joel Villatoro, KU Leuven}

\address{%
KU Leuven Departement Wiskunde,
Celestijnenlaan 200B, 
3001 Leuven, Belgium
and
Institut für Mathematik, 
Haus 9, Karl-Liebknecht-Straße 24, 
14476 Potsdam, Germany}


\begin{abstract}
We give a new construction of the holonomy and fundamental groupoids of a singular foliation. In contrast with the existing construction of Androulidakis and Skandalis, our method proceeds by taking a quotient of an infinite dimensional space of paths. This strategy is a direct extension of the classical construction for regular foliations and mirrors the integration of Lie algebroids via paths (per Crainic and Fernandes). In this way, we obtain a characterization of the holonomy and fundamental groupoids of a singular foliation that more clearly reflects the homotopic character of these invariants. As an application of our work, we prove that the constructions of the fundamental and holonomy groupoid of a foliation have functorial properties.
\end{abstract}

\maketitle

\tableofcontents

\section*{Introduction}
A regular foliation, is an integrable subbundle $D \into TM$ of the tangent bundle of a smooth manifold. The Frobenius theorem tells us that this is equivalent to the partition of $M$ into immersed submanifolds which satisfies a certain local normal form.

Regular foliations appear frequently in differential geometry. For example: the fibers of a submersion or the orbits of some particularly nice Lie group actions. However, it is often the case that one must consider situations where the action is not nice or the map is not quite a submersion. This requires the consideration of distributions which may change in rank and therefore a suitable generalization of regular foliations.

Along the lines of previous authors~\cite{AndrSk}\cite{DEB1}\cite{Hermann}\cite{SylvainArticle}\cite{Stefan}\cite{RoyThesis}, we define a foliation to be a sheaf of Lie algebras which is also a submodule of the sheaf of vector fields satisfying some finiteness conditions (Definition~\ref{defi:sing.fol}).

Singular foliations are closely related to the study of Lie algebroids. From a given Lie algebroid, one can construct a singular foliation by considering the sheaf of vector fields which arise as the image of the anchor map. On the other hand, given a leaf of a singular foliation, one can define a transitive Lie algebroid~\cite{AZ1}.

Two of the most useful tools for studying regular foliations are the associated fundamental groupoids and holonomy groupoids. These groupoids are invariants of the foliation which reflect the global properties of the leaves. When the foliation is regular it can be regarded as a Lie algebroid and both of these groupoids constitute integrations of this Lie algebroid structure.

For the singular case, Androulidakis and Skandalis~\cite{AndrSk} constructed the holonomy groupoid. Although it corresponds with the standard object in the regular case, the construction bears little resemblance to the classical method for defining this invariant. The goal of this paper is to provide an alternative construction of the fundamental groupoid and holonomy groupoid of a singular foliation. Our method admits clear analogies with the classical constructions as well as with the integration of Lie algebroids per Crainic and Fernandes~\cite{CrFeLie}.

Another path-based approach to constructing the fundamental groupoid of a singular foliation appears in a paper~\cite{SylvainArticle} by Laurent-Gengoux, Lavau and Strobl and later in \cite{LR19} by Laurent-Gengoux and Ryvkin. In ~\cite{SylvainArticle} work, we see that if the foliation arises as the image of a vector bundle, then one can imitate the Crainic and Fernandes construction to obtain a topological groupoid which integrates each leaf-wise algebroid.

In this paper we are faced with two distinct kinds of generalizations of manifolds: infinite dimensional mapping spaces and singular quotients. It is therefore convenient to choose a formalism of generalized manifold which encompasses both possibilities. The language of diffeology is suitable for this task without requiring much significant technical development. Therefore, this is the framework we will use for describing our manifold-like objects.
\subsection*{Constructing the holonomy groupoid}
In Section~\ref{subsec:holonomy}, we explain our construction of the holonomy groupoid. We define an $\CF$-path to be a time-dependent element $X(t)$ of $\CF$ together with an integral curve $\gamma(t)$. Using the flow of $X$, one can obtain the germ of a diffeomorphism from a transversal $T_0$ through $\gamma(0)$ to $T_1$ through $\gamma(1)$.
However, there is ambiguity in this construction. So the `holonomy' is actually the germ of an element of $\Diff(T_0,T_1)$ up to an equivalence relation. This equivalence relation is precisely the one used by Androlidakis and Zambon in \cite{AZ2}. The holonomy groupoid is therefore obtained by identifying $\CF$-paths which are holonomic.

In Section~\ref{section:compareas} we show that our construction is diffeomorphic (as a diffeological groupoid) to the previous construction by Androulidakis and Skandalis. Although the proof involves diffeological spaces and is somewhat involved, it does not require more than a basic understanding of the subject of diffeology.

\subsection*{Constructing the fundamental groupoid}
In Section~\ref{subsec:homotopy} we construct the fundamental groupoid of a singular foliation. Taking inspiration from Crainic-Fernandes style integration, the key notion is that of $\CF$-homotopy of $\CF$-paths. The fundamental groupoid of $\CF$ is therefore defined to be the set of $\CF$-paths up to $\CF$-homotopy. It turns out that two $\CF$-paths are holonomic whenever they are homotopic. This way, one can view the holonomy groupoid as a quotient of the fundamental groupoid, which correctly reflects the classical relationship between the two objects.

As we mentioned before, given a leaf $L$ of a singular foliation one can naturally define an integrable Lie algebroid $A_L \to L$. In Section~\ref{section:crainicfernandes}, we show that (leaf-wise), the fundamental groupoid of a singular foliation is the same as the source simply connected integration of this algebroid. This clarifies the close connection between the notion of $\CF$-homotopy and $A$-homotopy.

\subsection*{Functoriality}
In the last section, we show that the construction of the fundamental groupoid is functorial in the same way that the path integration of a Lie algebroid is functorial. However, in order to state the result, we are forced to develop a suitably general notion of morphism of foliated manifolds. To do this, we use the notion of a comorphism of sheaves of modules previously discussed by Higgins and Mackenzie~\cite{hm1993}.
\subsection*{Acknowledgements}
The authors would like to thank Marco Zambon for his advice throughout this project. We would also like to thank Rui Fernandes for his suggestions and comments. We would also like to thank Dan Christensen for his helpful remarks on an earlier draft. This work was partially funded by FWO EOS project G0H4518N and FWO Project G083118N.
\section{Foliations}
Let $\Cinf$ denote the sheaf of smooth functions on a manifold $M$.
In this text, a singular foliation on $M$ is a well-behaved $\Cinf$-submodule $\CF$ of the sheaf of vector fields $\SX$. The two key conditions are that $\CF$ should be locally finitely generated and involutive. Throughout, $M$ will denote a fixed smooth manifold.
\begin{defi}[Locally Finitely Generated]\label{defi:finitely.gen}
A sheaf of $C^\infty_M$-modules $\CF$ is said to be \emph{ locally finitely generated} if for all $x \in M$, there exists an open set $x \in U \subseteq M$, a natural number $n \ge 0$ and a surjective morphism of $C_M^\infty(U)$-modules:
\[ C_M^\infty(U)^n \twoheadrightarrow \CF(U) \]
\end{defi}

This definition is equivalent to the existence of a finite set of sections $s_1 , \ldots , s_N \in \CF(U)$ which generate $\CF(U)$ over $C^\infty_M(U)$.
\begin{defi}[Involutive]\label{defi:involutive}
A $C^\infty_M$ submodule $\CF \into \SX$ of the sheaf of vector fields on $M$ is said to be \emph{ involutive} if for each open set $U$, the image of the inclusion $\CF(U) \into \SX(U)$ is a Lie subalgebra.
\end{defi}
\begin{defi}[Foliation]\label{defi:sing.fol}
A \emph{(singular) foliation} on a manifold $M$ is a locally finitely generated and involutive $C^\infty$-submodule $\mathcal{F} \into \SX$. A manifold equipped with such an object is called a \emph{ foliated manifold}.
\end{defi}
By Hermann~\cite{Hermann}, Stefan~\cite{Stefan} and Sussman~\cite{Sussmann}, singular foliations partition $M$ into immersed submanifolds (called leaves) which are generated by the flows of vector fields inside of the foliation. The singular distribution obtained from the tangent spaces to the leaves is called the characteristic distribution of $\CF$.

Although we do not assume that $\CF$ is projective (i.e. the sheaf of sections of some vector bundle), we can still define something which resembles a vector bundle:
\begin{defi}\label{defi:localization}
Let $\CF$ be a $\Cinf$-module and let $x \in M$ be a point. We denote the ideal of smooth functions which vanish at $x \in M$ by $I_x$. The \emph{fiber of $\CF$ at $x$} is:
\[ A(\CF)_x := \frac{\CF(M)}{I_x \CF(M)} \]
Since $C^\infty_M(M) / I_x$ is canonically isomorphic to $\mathbb{R}$, the fiber of $\CF$ at $x$ is a real vector space. The \emph{fiberspace of $\CF$} is the disjoint union of the fibers:
\[ A(\CF) := \bigsqcup\limits_{x \in M} A(\CF)_x \]
\end{defi}
Many properties of $\CF$ can be determined from $A(\CF)$. For instance, $\CF$ is finitely generated in a neighborhood of $x \in M$ if and only if $A(\CF)_x$ is a finite dimensional vector space. Furthermore, if $\CF$ is locally finitely generated, then $\CF$ is projective if and only if the dimension of $A(\CF)_x$ is independent of $x \in M$.
\section{Diffeology}
Since we are interested in considering time dependent sections of a singular foliation, it is important to make sense of the notion of a smooth map into $\CF(U)$. We will approach this problem by using the language of diffeology.
\subsection{Basics}
We will give a brief overview of the main concepts from diffeology that we require. Diffeological spaces were first introduces by Souriau in~\cite{SouriauDiff1}\cite{SouriauDiff2}. For a more thorough treatment, we refer the reader to \cite{Diffeology}.
\begin{defi}[Diffeological Space]\label{defi:diffeo.sp}
A \emph{ diffeology} on a set $X$ is an assignment to any $d \in \mathbb{N}$ and a non-empty open $U \subseteq \mathbb{R}^d$, a subset $\mathcal{D}(U) \subseteq X^U$ of elements called \emph{ plots} such that:
\begin{enumerate}[(a)]
\item if $f \in X^U$ is constant, then $f \in \mathcal{D}(U)$;
\item for all $V \subseteq \mathbb{R}^n$, $f \in \mathcal{D}(U)$ and $g \in C^\infty(V,U)$, then $f \circ g \in \mathcal{D}(V)$;
\item if $f \in X^U$ and there exists an open cover $\{ U_i \}_{i \in I}$ of $U$ such that $f|_{U_i} \in \mathcal{D}(U_i)$ for all $i \in I$, then $f \in \mathcal{D}(U)$.
\end{enumerate}
A \emph{ diffeological space} is a pair $(X,\mathcal{D})$ where $X$ is a set and $\mathcal{D}$ is a diffeology on $X$. Given diffeological spaces $(X, \mathcal{D})$ and $(Y, \mathcal{D}')$, a function $f \colon X \to Y$ is called \emph{ smooth} if $f \circ \mathcal{D} \subseteq \mathcal{D'}$.
The set of smooth functions from $X$ to $Y$ will be denoted $C^\infty(X,Y)$.
\end{defi}
Any smooth manifold is canonically a diffeological space and a function between smooth manifolds is smooth in the diffeological sense if and only if it is smooth in the traditional sense.
\begin{defi}[\cite{Diffeology}(I.46)]
Suppose $X$ and $Y$ are diffeological spaces and $f \colon X \to Y$ is smooth. Then $f$ is called a \emph{ subduction} if for all plots $\phi \colon U \to Y$ and points \( p \in U \), there exists an open neighborhood $p \in U' \subseteq U$ and a plot $\til \phi \colon U' \to X$ such that:
\begin{equation}\label{eqn:lift}
f \circ \til \phi = \phi|_{U'}
\end{equation}
A \emph{ local subduction} is the case where the lift can arranged to pass through a specified point. That is, for all $x \in X$, given a plot $\phi \colon U \to Y$ such that $\phi(v) = f(x)$ for some $v \in U$, there exists an open neighborhood $U' \subseteq U$ of $v$ and a lift $\til \phi \colon U' \to X$ such that $\til \phi(v) = x$ and Equation~(\ref{eqn:lift}) holds.
\end{defi}
If $X$ and $Y$ are smooth manifolds, then $f \colon X \to Y$ is a local subduction if and only if it is a submersion. Now let us look at a few examples of diffeological spaces.
\begin{ex}[Intersection~\cite{Diffeology}(I.22)]
Suppose $\{ \mathcal{D}_i \}_{i \in I}$ is a set of diffeological structures on $X$ indexed by $I$. Then one can define the intersection diffeology
\[ \left( \bigcap\limits_{i \in I} \mathcal{D}_i \right)(U) := \bigcap\limits_{i \in I} \mathcal{D}_i(U) \]
\end{ex}
\begin{ex}[Products~\cite{Diffeology}(I.55)]
Suppose \( X \) and \( Y \) are diffeological spaces. Then the product diffeology on \( X \times Y \) is defined to be the intersection of all diffeologies on \( X \times Y \) such that for all plots \( \phi \colon U_1 \to X \) and \( \psi \colon U_2 \to Y \) the function \( \phi \times \psi \colon U_1 \times U_2 \to X \times Y \) is a plot.
\end{ex}
\begin{ex}[Mapping Spaces\cite{Diffeology}(I.13)]
Let $X$ and $Y$ be diffeological spaces. Then $C^\infty(X,Y)$ is also a diffeological space. A function $f \colon U \to C^\infty(X,Y)$ is a plot if and only if the associated function $U\times X\fto Y$ given by $(u, x) \mapsto f(u)(x)$ is smooth.
\end{ex}
\begin{ex}[Quotient Diffeology\cite{Diffeology}(I.50)]
Suppose $X$ is a diffeological space and $\sim$ is an equivalence relation on $X$. Let $\pi \colon X \to X/ \mathtt{\sim}$ denote the quotient function. Then the quotient diffeology on $X/\mathtt{\sim}$ is the intersection of all diffeologies on $X/\mathtt{\sim}$ for which $\pi \colon X \to X/\mathtt{\sim}$ is diffeological morphism.

This is equivalent to saying that the diffeology on $X/\mathtt{\sim}$ is the unique one which makes $\pi$ a subduction.
\end{ex}
\begin{ex}[Subset Diffeology\cite{Diffeology}(I.33)]
Suppose $Y \subseteq X$ and $X$ is equipped with a diffeology and let $\iota \colon Y \to X$ be the inclusion map. Then we can define a plot on $Y$ to be any function $f \colon \mathbb{R}^n \to Y$ such that $\iota \circ f$ is a plot on $X$.

If $X$ is a smooth manifold and $Y \into X$ is an immersed submanifold. Then the diffeology coming from the smooth structure on $Y$ agrees with the subset diffeology on $Y$ if and only if $Y$ is an initial submanifold.
\end{ex}
\begin{ex}[Sections of a vector bundle]
Suppose $E \to M$ is a smooth vector bundle over a manifold $M$. Let $\Gamma_E$ denote the sheaf of smooth sections of $E$. We have an inclusion $\Gamma_E(M) \subseteq C^\infty(M,E)$. Therefore, $\Gamma_E(M)$ inherits a subset diffeology from the mapping space diffeology on $C^\infty(M,E)$. From now on, we will take the diffeology on the sections of a vector bundle to be implicit.
\end{ex}
\subsection{The coefficient diffeology}
We wish to generalize the last example to sections of a $\Cinf$-module. The basis of the definition will be that we want the smallest diffeology with the following properties: It agrees with the diffeology on sections of a vector bundle, it makes module homomorphisms smooth, and is compatible with the restriction and gluing sheaf operations. To our knowledge this diffeology has not appeared before in the literature.
\begin{defi}[Coefficient Diffeology]\label{defi:coef.diffeo}
Let $\CF$ be a $\Cinf$-module. Given a subset of some Euclidean space $U$, a function $\phi \colon U \to \CF(M)$ is a plot if and only if it locally factors through the sections of a vector bundle.

In other words: $\phi$ is a plot, if and only if for all $u_0 \in U$ and $p_0 \in M$, there exist open neighborhoods $U' \subseteq U$ and $V \subseteq M$ of $u_0$ and $p_0$ together with:
\begin{itemize}
\item a vector bundle $E \to V$,
\item a morphism of $\Cinf$-modules $T \colon \Gamma_E(V) \to \CF(V)$,
\item and a smooth function $\til \phi \colon U' \to \Gamma_E(V)$
\end{itemize} 
such that the following diagram commutes:
\[\begin{tikzcd}[column sep = large]
& \Gamma_E(V) \arrow{dr}{T}  \\
U' \arrow[dashed]{ur}{\exists \til \phi} \arrow{r}{\phi|_{U'}} & \CF(M) \arrow{r} & \CF(V)
\end{tikzcd}\]
This is called the \emph{ coefficient diffeology} on $\CF(M)$.
\end{defi}
The next lemma gives another characterization of the coefficient diffeology.
\begin{lem}
Suppose $\CF$ is a sheaf of $\Cinf$-modules. Let $\CF(M)$ be equipped with the coefficient diffeology. Then, $\phi \colon U \to \CF(M)$ is a plot if and only if for all $u_0 \in U$ and $p_0 \in M$ there are open neighborhoods $U' \subseteq U$ of $u_0$ and $V \subseteq M$ of $p_0$, together with
\begin{itemize}
\item a finite number of functions:
\[ c^1, \ldots , c^n \in C^\infty(U' \times V, \mathbb{R}) \]
\item and elements $X_1 , \ldots , X_n \in \CF(V)$
\end{itemize} 
such that:
\begin{equation}\label{eqn:coeffdif}
\forall u \in U' , p \in V \quad \phi(u)_p =\sum\limits_{i=1}^n c^i(u,p) (X_i)_p
\end{equation}
\end{lem}
\begin{proof}
If we are given a $\phi$ as in the statement of the lemma. We want to show that $\phi$ is smooth. Supose around $u_0$ and $p_0$ we have functions $c^i$ and elements $X_i \in \CF$ as in the statement of the lemma. Then the following choices witness the fact that $\phi$ is smooth: 
\begin{itemize}
\item $E := \mathbb{R}^n \times V$ is the trivial vector bundle of rank $n$
\item $T \colon\Gamma_E(V) \to C^\infty(M)$ is the module homomorphism associated to the choice of elements $X_1, \ldots , X_n$ and
\item $\til \phi$ is the smooth function:
\[ \til \phi \colon U' \to \Gamma_E(V) \qquad  \til \phi(u)_p := (c^1(u,p) , \ldots , c^n(u,p), p) \]
\end{itemize}
On the other hand, for any vector bundle $E$ there exists another vector bundle $W$ such that $W \oplus E$ is trivial (Theorem 3.3 in \cite{Hirsch}). Therefore, Definition~\ref{defi:coef.diffeo} is unchanged if we require that $E$ is the trivial bundle. So $\phi$ must locally factor through a trivial bundle. The functions $\{ c^i \}_{i=1}^n$ in Equation~(\ref{eqn:coeffdif}) correspond to the coefficients of $\til \phi$ relative to a choice of trivialization.
\end{proof}
The coefficient diffeology coincides with the mapping space diffeology when $\CF$ is the sections of a vector bundle. It also makes all $C^\infty_M$-module homomorphisms smooth. Although the definition of the coefficient diffeology is stated in terms of global sections, we can define the diffeology on sections over an arbitrary open subset similarly.

In the following example, we see that when $\CF$ is a submodule of the sections of a vector bundle, then the coefficient diffeology may not coincide with the subset diffeology.
\begin{ex}\label{ex:bad.diffeology}
Let $M = \mathbb{R}$ and $\CF = \langle f \frac{\partial}{\partial t} \rangle \le \SX(\mathbb{R})$ where $f \colon \mathbb{R} \to \mathbb{R}$ is a smooth function such that $f(x) = 0$ for $x \le 0$ and $f(t) > 0$ when $x > 0$. Consider the smooth function $X \colon \mathbb{R} \to \CF(\mathbb{R})$ such that
\[ X(t)(x_0) = f(x_0- t^2) \left(\frac{\partial}{\partial x}\right)_{x_0}\]
Then $X$ fails to be smooth relative to the coefficient diffeology on $\CF$. On the other hand, $X$ is smooth as a time-dependent vector field, i.e. as a smooth map $\mathbb{R} \to \SX(\mathbb{R})$.
\end{ex}
From now on, all $C^\infty_M$-modules shall be implicitly equipped with the coefficient diffeology.
\subsection{Diffeology and the fiberspace}
An important application of the coefficient diffeology is that it can be used to define a diffeology on the fiberspace. The contents of this subsection is not needed to understand the constructions in Section~\ref{sec:fpaths} where we construct the holonomy groupoid and fundamental groupoid of a singular foliation. However, it will be crucial for understanding Sections \ref{section:crainicfernandes} and \ref{section:functoriality} where we study the relationships of these constructions to algebroid homotopy and their functoriality properties, respectively.
\begin{defi}\label{defi:fiberdiffeology}
Given a $\Cinf$-module $\CF$ recall the fiberspace $A(\CF)$ from Definition~\ref{defi:localization}. Since the fiber of $\CF$ at $x$ is obtained from an equivalence relation on $\CF(M)$, the \emph{ evaluation map} is the surjective function:
\[ \ev \colon \CF(M) \times M \to A(\CF) \]
\[ (X,p) \mapsto \ev_p(X) \]
which sends $(X,p)$ to the class of $X$ in $A(\CF)_p$. The \emph{ diffeology on $A(\CF)$} is the unique diffeology which makes the evaluation map a subduction.
\end{defi}
\begin{rem}
There is a slight difference in notation here between the evaluation map that appears in \cite{AndrSk} and the evaluation map just defined. Note that the evaluation map used by Androulidakis and Skandalis is only well defined for a foliation whereas the definition above makes sense for an arbitrary sheaf of modules.
\end{rem}
The diffeology on the fiberspace makes $\pi \colon A(\CF) \to M$ into a local subduction. If we think of $\CF(M) \times M$ as a trivial vector bundle of infinite rank, the rescaling, addition, and zero section operations on $\CF(M) \times M$ descend to smooth maps.

In the terminology of Christensen and Wu\cite{christensen2020exterior} the set $A(\CF)$ is a diffeological vector pseudo-bundle\footnote{In some articles~\cite{tngtdiffeo} the term \emph{vector space} is also used for the same object}. The pseudo- prefix is motivated by the fact that diffeological vector bundles as defined here need not be locally trivial. We will include diffeological vector pseudo-bundles here but we will not really need it until Section~\ref{section:functoriality}.
\begin{defi}
A \emph{diffeological vector pseudo-bundle} over a diffeological space \( X \) consists of a smooth function\( \pi \colon E \to X \) of diffeological spaces together with three smooth functions:
\[ E \times_M E \to E \qquad (v,w) \mapsto v + w, \qquad  \mathbb{R} \times E \to E  \qquad (\epsilon , v ) \mapsto \epsilon \cdot v, \qquad  M \to E \qquad x \mapsto 0_x \]
which make each fiber of \( \pi \) into an \( \mathbb{R} \)-vector space. A \emph{morphism} of diffeological vector pseudo-bundles \( F \colon E \to W \) covering \( f \colon X \to Y \) is a smooth function such that \( \pi \circ F = f \circ \pi \) and \( F \) is fiber-wise linear.
\end{defi}
In \cite{AZ1} it was observed that the restriction of the fiberspace to a leaf is a Lie algebroid. The next lemma tells us that this Lie algebroid structure agrees with the associated subspace diffeology.
\begin{lem}\label{lem:ALisalgebroid}
Suppose $\CF \into \SX_M$ is a foliation on $M$ and $L \into M$ is a leaf. Then the set:
\[ A(\CF)_L := \pi^{-1}(L) \subseteq A(\CF) \]
is a Lie algebroid over $L$ when equipped with the subset diffeology.
\end{lem}
\begin{proof}
Since the leaf of a singular foliation is an initial submanifold. It follows that the smooth structure on $L$ is the subset diffeology relative to the inclusion $L \into M$. This ensures that the projection $\pi|_L \colon A(\CF)_L \to L$ is smooth.

We will now show that the subset diffeology on $A(\CF)_L$ makes it a smooth manifold by exhibiting local trivializations. Suppose $p \in L$ is an arbitrary point in the leaf in question. Let $X_1, \ldots , X_n$ be a set of sections of $\CF(M)$ which represent a basis for $A(\CF)_p$.
Nakayama's Lemma implies that $X_1 , \ldots , X_n$ are a minimal set of generators in some open neighborhood $U$ of $p$.

Now consider the following smooth function:
\[ \phi \colon \mathbb{R}^n \times U \to A(\CF)_{U} \]
\[ (c^1, \ldots , c^n, q) \mapsto \sum_{i} c^i \ev_q(X_i) \]
This function fits into a commuting square:
\[
\begin{tikzcd}
C^\infty_M(U)^n \times U \arrow[two heads]{d}{\ev} \arrow[two heads]{r} & \CF(U) \times U \arrow[two heads]{d}{\ev} & \\
\mathbb{R}^n \times U \arrow{r}{\phi} & A(\CF)_U
\end{tikzcd}
\]
Since $X_1, \ldots , X_n$ generate $\CF(U)$, the top arrow is a subduction. The evaluation maps are also subductions. Therefore $\phi$ is a surjective subduction. On the other hand, since the rank of $A(\CF)_L$ is constant, the restriction
\[ \phi|_{U \cap L} \colon \mathbb{R}^n \times (U \cap L) \to A(\CF)_{U \cap L} \]
is a bijective subduction (a diffeomorphism of diffeological spaces).
To see the algebroid structure, let $S \colon \CF(M) \to \Gamma_{A(\CF)_L}(L)$ denote the function which sends an element $X \in \CF(M)$ to the associated section of $A(\CF)_L$. The bracket on sections of $A(\CF)_L$ and the anchor map $A(\CF)_L \to TL$ are uniquely determined by the following equations:
\[ [ S(X), S(Y)]_{A(\CF)} = S([X,Y]_{\CF}) \qquad \rho(\ev_p(X)) = X_p \in T_p L \]
\end{proof}
\section{\texorpdfstring{$\CF$}{CF}-paths}\label{sec:fpaths}
From now on we assume that $\CF$ is a singular foliation and that $\CF(M)$ is equipped with the coefficient diffeology.
\begin{defi}
Given an open set $U \subseteq M$. A smooth function:
\[ X \colon [0,1] \to \CF(U) \qquad t \mapsto X(t) \]
is called a \emph{ time-dependent element of $\CF(U)$}. Given such a time-dependent element, the notation $\Phi^t_X \colon U \to U$ will denote the flow of $X$ (when it exists).
\end{defi}
Closely related to the notion of a time-dependent element is the notion of an $\CF$-path. This is the main object which we will use in our construction of the holonomy and fundamental groupoids.
\begin{defi}[$\CF$-Path]\label{defi:f.path}
An \emph{ $\CF$-path } is a pair $(X,p)$ where $X$ is a time-dependent element of $\CF(M)$ and $p$ is an integral curve of $X$.
The set of all $\CF$-paths will be denoted by $\CP(\CF)$.

The \emph{ source} of $(X,p)$ is defined to be $p(0) \in M$ while the \emph{ target} is $p(1)$. The source and target maps define subductions $\s,\t \colon \CP(\CF) \to M$.
\end{defi}
The next lemma tells us that the flow of a time-dependent element of $\CF$ is compatible with the coefficient diffeology on $\CF(M)$.
\begin{lem}\label{lem:preserves.foliation}
Suppose $(X,p_0)$ is an $\CF$-path and the flow $\Phi^t_X$ of $X$ exists for all $t \in [0,1]$. Then
\[ (\Phi^t_X)_* \colon \SX(M) \to \SX(M) \]
satisfies $(\Phi^t_X)_* (\CF_M) = \CF_M$ and the following function is smooth:
\[ [0,1] \times \CF(M) \to \CF(M) \qquad (t,Y) \mapsto (\Phi^t_X)_* Y\]
\end{lem}
\begin{proof}{}
For this proof we use the theory of derivations on a vector bundle (sometimes called covariant differential operators)~\cite{MK2}. The general strategy of the proof is inspired by the similar proof (for constant vector fields) in \cite{AndrSk} and \cite{AutOri}. The property that we wish to show is local in $M$. Therefore, we can assume without loss of generality that $\CF(M)$ is finitely generated.

Suppose $\{ X_i \}_{i=1}^n \subseteq \CF(M)$ is a set of generators. Let $\SR^n$ be the trivial vector bundle of rank $n$ and let $\{ e_i \}_{i=1}^n$ be the canonical basis. Define $\Psi \colon \SR^n \to TM$ to be the unique vector bundle morphism such that $\Psi(e_i) = X_i$.

Since $X \colon [0,1] \to \CF(M)$ is smooth and $\CF$ finitely generated, for each $ 1 \le i,k \le n$ the definition of the coefficient diffeology tells us that there exist smooth functions $a^k_{i} \colon [0,1] \to \Cinf(M)$ such that:

\[ [X(t), X_i ] = \sum_{k=1}^n a^k_i(t) X_k \]
Let $(D^t, X(t))$ be the unique time-dependent derivation on $\SR^n$ such that:
\[ D^t( e_i ) = \sum_{k=1}^n a^k_i(t) e_k \]
Let $\Phi_{(D,X)}^t \colon \SR^n \to \SR^n$ denote the flow of this derivation.

The flow $\Phi^t_X \colon M \to M$ can be thought of as the flow of the time-dependent derivation $([X, \cdot ] , X)$ and from the definition of $D$ it follows that:
\begin{equation}\label{eqn:dercompat}
\forall e \in \Gamma(\SR^n) , \qquad \Psi \circ D(e ) = [X , \Psi(e)]
\end{equation}
We can now finish the proof. For any $Y \in \CF(M)$, there exists a lift $\til Y \in \Gamma(\SR^n)$. By Equation~(\ref{eqn:dercompat}), we have that:
\[ (\Phi^t_X)_* Y = \Psi \circ (\Phi^t_{(D,X)})_* \til Y \]
Since the image of $\Psi$ is $\CF(M)$ it follows that $(\Phi^t_X)_* Y \in \CF(M)$. Furthermore, the right hand side gives a factorization of $t \mapsto (\Phi^t_X)_* Y $ through a vector bundle. This implies that $t \mapsto (\Phi^t_X)_* Y $ is smooth and therefore $(t, Y) \mapsto (\Phi^t_X)_* Y$ is smooth.
\end{proof}
The set of $\CF$-paths has a natural monoidal structure. Let us fix a reparameterization of the interval $r \colon [0,1] \to [0,1]$ such that $r$ is constant in an open neighborhood of the endpoints.
\begin{defi}\label{defi:concat}
Suppose $(X,p)$ and $(Y,q)$ are $\CF$-paths such that $p(1) = q(0)$. Then we can define the concatenation:
\[ (Y,q) \odot (X,p) := (Y \odot X, q \odot p) \]
\[(Y \odot X) (t) := \begin{cases} 2 r'(2t) X(r(2t)) & t \in [0, \frac{1}{2}] \\ 2 r'(2t-1)Y(r(2t-1)) & t \in (\frac{1}{2},1]
\end{cases}\]
\[
(q \odot p)(t) := \begin{cases} p(r(2t)) & t \in [0, \frac{1}{2} ] \\ q(r(2t-1)) & t \in (\frac{1}{2}, 1] \end{cases} \]
\end{defi}
\begin{rem}\label{rmk:flowvsconcat}
Concatenation is compatible with flow in the following sense:
Given an arbitrary $\CF$-path $(X,p)$, we know that the time-one flow of $X$ exists in a neighborhood of $p(0)$. Furthermore, given two $\CF$-paths $(X, p)$ and $(Y, q)$, then:
\[ \Phi_{Y \odot X}^1 = \Phi^1_Y \circ \Phi^1_X \]
wherever the relevant flows exist. This means that the monoidal structure on $\CF$-paths is compatible with the groupoid of germs of diffeomorphisms over $M$. \end{rem}
The goal for the rest of this section will be to define two equivalence relations on $\CP(\CF)$, called holonomy and homotopy, for which the concatenation descends to an honest (diffeological) groupoid structure on the equivalence classes.
\subsection{Holonomy}\label{subsec:holonomy}
The notion of holonomy for classical (i.e. regular) foliations comes from constructing germs of diffeomorphisms between slices (i.e transverse submanifolds) out of paths tangent to the leaves.
In the singular case, this approach is complicated by the increased ambiguity when choosing slices. Following the example of Androulidakis and Zambon~\cite{AZ2}, we correct this problem by operating modulo a subgroup of diffeomorphisms which control this ambiguity.
\begin{defi}[Slices]
Suppose $(M,\CF)$ is a foliated manifold.
A \emph{ slice through $x \in M$} is an embedded submanifold $S \hookrightarrow M$, containing $x$, such that for all $y \in S$, we have that $T_y M = T_y S + T_y L_y$ where $L_y$ is the leaf containing $y$ and $T_x S \cap T_x L_x = 0$.
\end{defi}
Now, for all $x \in M$, let us choose a slice $S_x$ through $x$. The transversality property for slices implies that each $S_x$ inherits a foliation $\CF_{S_x}$ by restriction of the foliation on $M$ to $S_x$\cite{AZ2}. Let $\Diff_\CF(S_x,S_y)$ denote the set of germs of (foliation preserving) diffeomorphisms $S_x \to S_y$ which map $x$ to $y$. Finally, $\exp(I_x \CF_{S_x}) \le \Diff(S_x, S_x)$ denotes the subgroup of germs of diffeomorphisms generated by the flows of (possibly time dependent) elements of $I_x \CF_{S_x}$.
\begin{defi}
The \emph{ holonomy transformation groupoid} is the set:
\[ \HT := \bigsqcup_{x,y \in M} \frac{\Diff_\CF(S_x, S_y)}{\exp(I_y \CF_{S_y})} \]
Elements of $\HT$ are called holonomy transformations.
\end{defi}
\begin{rem}
The holonomy transformation groupoid is not a totally canonical object since it depends on the choice of a slice $S_x$ at every point $x \in M$. When $\CF$ is regular, the foliation on each slice is trivial and the holonomy transformation groupoid can be identified with the groupoid of germs of diffeomorphisms between slices.
\end{rem}
Our next step is to define a function $\Hol \colon \CP(\CF) \to \HT$ which will associate a holonomy transformation to every $\CF$-path. To state the definition, we will need to borrow Lemma A.8 from \cite{AZ2}.
\begin{lem}\label{lem:slice}
Suppose $S$ and $S'$ are two slices through $x \in M$. Then there exists a vector field $X \in I_x \CF$ and an open neighborhood $U \subseteq M$ of $x$ such that:
\[ U \cap \Phi_X^1 (S) = U \cap S' \]
\end{lem}
Now suppose we are given an $\CF$-path $(X,p)$. Then, Lemma~\ref{lem:preserves.foliation}, implies that $\Phi_X^1(S_{p(0)})$ is a slice through $p(1)$.
Furthermore, Lemma~\ref{lem:slice}, tells us that there exists a vector field $Z \in I_{p(1)} \CF$ and a neighborhood $U \subseteq M$ of $p(0)$ such that:
\[ \Phi_Z^1 \circ \Phi_X^1 (S_{p(0)} \cap U ) \subseteq S_{p(1)} \]
Since $\Phi^1_Z \circ \Phi^1_X$ is a local diffeomorphism around $p(0)$ it defines an element of $\Diff_\CF(S_{p(0)} , S_{p(1)})$ which represents an element of $\HT$.
This forms the basis for our definition of holonomy.
\begin{defi}\label{defi:holonomy}
Let $(X,p)$ be a $\CF$-path. Suppose $Z$ is a time-dependent element of $I_{p(1)} \CF$ such that the function
\[ \Phi_Z^1 \circ \Phi_X^1|_{S_{p(0)}} \colon S_{p(0)} \to S_{p(1)} \]
is well defined in a neighborhood of $p(0) \in S_{p(0)}$.
The \emph{ holonomy of $(X,p)$}, denoted $\Hol(X,p)$, is the equivalence class of $\Phi_Z^1 \circ \Phi_X^1$ in
\[ \frac{\Diff_\CF(S_{p(0)} , S_{p(1)} )}{\exp( I_{p(1)} \CF_{S_{p(1)}})} \subseteq \HT \]
The \emph{ holonomy equivalence} relation is the equivalence relation generated by the fibers of $\Hol$.
\end{defi}

\begin{figure}
\centering
\includegraphics[scale=1.25]{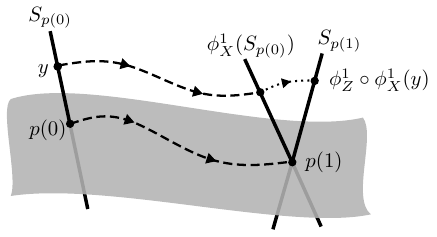}
\caption{A visualization the holonomy of an $\CF$-path $(X,p)$. The dashed lines represent integral curves of $X$ while the dotted line represents an integral curve of $Z$.}
\label{holonomy_visualization}
\end{figure}

\begin{thm}\label{thm:holwd}
The function $\Hol \colon \CP(\CF) \to \HT$ is well-defined.
\end{thm}
To prove this theorem, we will make use of another lemma from \cite{AZ2}:
\begin{lem}\label{lem:slice2}
Suppose $X \colon [0,1] \to \CF(M)$ is a time-dependent element of $I_x \CF$ such that $\Phi^1_X (S_x) \subseteq S_x$. Then there exists another time-dependent element $Y \colon [0,1] \to I_x \CF_{S_x}$ of $I_x \CF_{S_x}$ such that $\Phi^1_X |_{S_x} = \Phi_{Y}^1$.
\end{lem}
See Lemma A.6 in \cite{AZ2} for the proof.
\begin{proof}(of Theorem~\ref{thm:holwd})
By Lemma~\ref{lem:slice} and Lemma~\ref{lem:preserves.foliation}, we know that for a given $\CF$-path $(X,p)$ it is always possible to find a $Z \in I_{p(1)} \CF$ which satisfies Definition~\ref{defi:holonomy}. Therefore, we only need to show the definition of $\Hol(X,p)$ does not depend on the choice of $Z$.

Suppose we have $Z$ and $Z'$ such that $\Phi_Z^1 \circ \Phi_X^1$ and $\Phi_{Z'}^1 \circ \Phi_X^1$ both yield well-defined functions $S_{p(0)} \to S_{p(1)}$. We are finished if we can show that the germ of
\begin{equation}\label{eqn:holonomy.wd}
(\Phi_{Z'}^1 \circ \Phi_X^1) \circ {(\Phi_Z^1 \circ \Phi_X^1)}^{-1} = \Phi_{Z'}^1 \circ (\Phi_Z^1)^{-1} \colon S_{p(1)} \to S_{p(1)}
\end{equation}
at $p(1)$ is an element of $\exp(I_{p(1)} \CF_{p(1)})$. Since $\Phi_{Z'}^1 \circ (\Phi_Z^1)^{-1}$ is the composition of two flows, there exists a time dependent element $\til Z$ of $\CF(M)$ which takes values in $I_{p(1)} \CF$ and such that $\Phi_{\til Z}^1 = \Phi^1_{Z'} \circ (\Phi^1_Z)^{-1}$.

By Lemma \ref{lem:slice2}, it follows that there exists a time-dependent element $\hat{Z}$ of $I_{p(1)} \CF_{p(1)}$ such that:
\[ \Phi_{\hat Z}^1 = \Phi_{\til Z}^1 = \Phi_{Z'}^1 \circ (\Phi_{Z}^1)^{-1} \]
which shows that the function (\ref{eqn:holonomy.wd}) is an element of $\exp( I_{p(1)} \CF_{p(1)})$.
\end{proof}
The argument in the preceding proof can adapted to establish a sufficient condition for trivial holonomy.
\begin{lem}\label{lem:trivial.holonomy}
Suppose $(X,p)$ is an $\CF$-path such that $X$ is a time-dependent element of $I_{p(0)} \CF$. Then the holonomy of $(X,p)$ is trivial. That is: $\Hol(X,p) = 1_{p(0)} \in \HT$.
\end{lem}
\begin{proof}
Let $Z$ be a time-dependent element satisfying Definition~\ref{defi:holonomy}. Then the holonomy of $(X,p)$ is the equivalence class of $\Phi_Z^1 \circ \Phi_X^1|_{S_{p(0)}}$ in $\HT$.
Since both $Z$ and $X$ take values in $I_{p(0)} \CF$, we know that $\Phi_Z^1 \circ \Phi_X^1$ is the flow of some time dependent element $\til Z$ of $I_{p(0)} \CF$.
By Lemma~\ref{lem:slice2} we can conclude that it is also the flow of a time dependent element of $I_{p(0)} \CF_{S_{p(0)}}$.
\end{proof}
\begin{thm}\label{thm:ind.slices}
The holonomy equivalence relation does not depend on the choice of slices used to construct $\HT$.
\end{thm}
\begin{proof}
Suppose we choose alternative slices $S'_x$ through each $x \in M$. Let $\HT'$ denote the holonomy transformation groupoid relative to these choices and $\Hol' \colon \CP(\CF) \to \HT'$ be the resulting holonomy map. By Lemma~\ref{lem:slice}, we know that for each $x \in M$ there exists a vector field $Z_x \in I_x \CF$ such that $\Phi^1_{Z_x}$ defines a germ of a diffeomorphism from $S_x$ to $S'_x$. Then for each $x,y \in M$ we get a bijection:
\[\Diff(S_x, S_y) \to \Diff(S'_x, S'_y) \qquad f \mapsto \Phi^1_{Z_y} \circ f \circ {(\Phi^1_{Z_x})}^{-1} \]
Since each $\Phi^1_{Z_x}$ preserves the foliation (Lemma~\ref{lem:preserves.foliation}), we obtain a bijection $\Psi \colon \HT \to \HT'$. For each $x \in M$, let $(Z_x, x)$ be the constant $\CF$-path. Then we have that for all $(X,p) \in \CP(\CF)$:
\[ \Hol'( X,p) = \Hol' \left( (Z_{p(1)}, p(1) ) \odot (X,p) \odot (-Z_{p(0)} , p(0)) \right) = \Psi \circ \Hol(X,p) \]
where we apply Lemma~\ref{lem:trivial.holonomy} for the first equality.
Therefore, the holonomy equivalence relations relative to $\Hol$ and $\Hol'$ are the same.
\end{proof}
\begin{defi}
The \emph{ holonomy groupoid}, denoted $\CH(\CF)$, is defined to be the set of holonomy equivalence classes of $\CP(\CF)$ equipped with the quotient diffeology.
\end{defi}
The set $\CH(\CF)$ is a groupoid over $M$ with product inherited from the concatenation operation on $\CF$-paths. Following Remark~\ref{rmk:flowvsconcat}, we know that this groupoid structure makes the injection $\CH(\CF) \into \HT$ a groupoid homomorphism. In this way, we can regard $\CH(\CF)$ as a (set-theoretic) subgroupoid of $\HT$.
\subsection{\texorpdfstring{$\CF$}{CF}-Homotopy}\label{subsec:homotopy}
To define the $\CF$-homotopy equivalence relation, we need to understand deformations of $\CF$-paths.
\begin{defi}
A \emph{ variation of $\CF$-paths} is a pair $(X,p)$ where $X$ and $p$ are smooth functions:
\[ X \colon [0,1]^2 \to \CF(M) \qquad p \colon [0,1]^2 \to M \]
such that for all $s_0 \in [0,1]$:
\[ (X|_{[0,1] \times \{ s_0 \}} , p|_{[0,1] \times \{ s_0 \}}) \in \CP(\CF) \]
and $p(0,s)$ does not depend on $s$.

Given such a variation $(X,p)$ we define the \emph{ complement} of $(X,p)$ to be the two-parameter vector field:
\[ \forall (t,s) \in [0,1]^2 \quad \forall x \in M \qquad Y(t,s)|_{\Phi^{t,s}_X (x)} := \left( \frac{\text{d}}{\text{d} u} \right)\bigg\rvert_{u = s} \Phi^{t,u}_X(x) \]
where $\Phi^{t,s}_X$ denotes the flow of $X$ in the $t$-direction.
\end{defi}
\begin{figure}[h]
\centering
\begin{tikzpicture}
\coordinate (TR) at (5,3);
\coordinate (BR) at (4.5,1);
\coordinate (L) at (0,2);
\coordinate (M) at (2.3,2);
\coordinate (MR) at (4.75,2);

\draw (L) to[out=45,in=140] (TR);
\draw (L) to[out=-45,in=180] (BR);
\draw (BR) to[out=90,in=270] (TR);

\filldraw[black] (L) circle (2pt) node[anchor=east] {$x$};
\filldraw[black] (BR) circle (2pt) node[anchor=west] {$\Phi^{1,0}_X(x)$};
\filldraw[black] (MR) circle (2pt) node[anchor=west] {$\Phi^{1,s}_X(x)$};
\filldraw[black] (TR) circle (2pt) node[anchor=west] {$\Phi^{1,1}_X(x)$};

\draw[->,black, thick] (M) -- ++(-10:0.5) node[anchor=north] {$X(t,s)$};
\draw[dashed] (L) to[out=40, in=170] (M) to [out=-10, in = 190] (MR);
\draw[->,black, thick] (M) -- ++(80:0.5) node[anchor=west] {$Y(t,s)$};
\end{tikzpicture}
\caption{An illustration of the relationship between a variation and its complement.}
\end{figure}
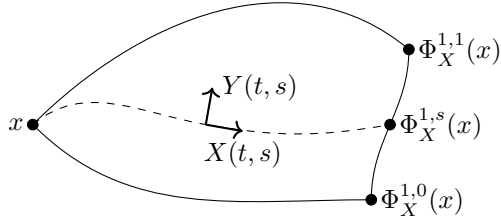
\begin{rem}
Since $p(\cdot, s)$ is an integral curve of $X(\cdot,s)$ for each fixed $s$, we know that $Y(t,s)|_{p(t,s)} = \frac{\text{d}}{\text{d} s} p(t,s)$. This implies that $p(t,s)$ is a classical boundary-preserving homotopy between the paths $p(t,0)$ and $p(t,1)$ if and only if $Y(1,s)|_{p(1,s)} = 0$ for all $s \in [0,1]$.
\end{rem}
\begin{lem}
Suppose $(X,p)$ is a variation of $\CF$-paths with complement $Y$. Then $Y \in C^\infty([0,1]^2,\CF(M))$.
\end{lem}
\begin{proof}{}
Let $(\Phi^{t,s}_X)_* \colon \SX(M) \to \SX(M)$ be the push-forward along the flow (in the $t$-direction) of $X$. The variation of parameters formula yields an expression for $Y(t,s)$ in terms of $\Phi_X^{t,s}$:
\begin{equation}\label{eqn:variation}
Y(t,s) = (\Phi^{t,s}_X)_* \int\limits_{0}^t {(\Phi^{u,s}_X)}_*^{-1} \frac{ \text{d} X}{\text{d} s} (u,s) \text{d} u
\end{equation}
By Lemma~\ref{lem:preserves.foliation} we know that $(\Phi^{t,s}_X)_*$ defines a smooth function $[0,1]^2 \times \CF(M) \to \CF(M)$. Furthermore, integrals of smooth functions $[a,b] \to \CF(M)$ are smooth since we can integrate in each coefficient separately. Therefore, we conclude that Equation~(\ref{eqn:variation}) defines a smooth function $Y \colon [0,1]^2 \to \CF(M)$.
\end{proof}
\begin{defi}[Homotopy]\label{defi:mon}
Suppose $(X_0,p_0)$ and $(X_1,p_1)$ are $\CF$-paths. Then $(X_0,p_0)$ is said to be \emph{ $\CF$-homotopic} to $(X_1,p_1)$ if and only if there exists a variation $(X,p)$, with complement $Y$ such that
\[ (X,p)|_{s=i} = (X_i, p_i) \quad \text{ for } \quad i=0,1 \quad \text{ and } \]
\[ Y(1,s) \in I_{p(1,s)} \CF \qquad \forall s \in [0,1].\]

The diffeological space of all $\CF$-homotopy classes is called the \emph{ fundamental groupoid} of $\CF$ and is denoted $\Mon(\CF)$
\end{defi}
\begin{rem}
Another notion of homotopy for singular foliations appears in an article of Laurent-Gengoux, Lavau and Strobl\cite{SylvainArticle} and more recently in~\cite{LR19} using the language of \( Q \)-manifolds. At the moment we do not know the precise relationship between these two definitions but we suspect that our notion of homotopy is essentially a `truncation' of theirs.
\end{rem}
\begin{thm}\label{thm:hol.to.hom}
Suppose $(X_0,p_0)$ and $(X_1,p_1)$ are $\CF$-homotopic, then they have the same holonomy.
\end{thm}
\begin{proof}
{}
Let $(X(t,s),p(t,s))$ be an $\CF$-homotopy from $(X_0,p_0)$ to $(X_1, p_1)$ with complement $Y(t,s)$. For convenience, let $q := p_0(1) = p_1(1)$. Let $\Phi_X^{t,s} \colon M \to M$ denote the flow of $X$ in the $t$-direction and let $\Psi_Y^{t,s}$ denote the flow of $Y$ in the $s$-direction. By assumption we have that:
\[ \Phi_X^{1,0} = \Phi_{X_0}^1 \quad \text{ and } \quad \Phi_X^{1,1} = \Phi_{X_1}^1 \]
On the other hand, we know that $\Phi_X^{1,s} = \Psi_Y^{1,s} \circ \Phi_X^{1,0}$ and we can therefore conclude that:
\[ \Psi^{1,1}_Y \circ \Phi_{X_0}^{1} = \Phi_{X_1}^{1} \]
By assumption, $Y(1,s)|_{q} \in I_{q} \CF$ and so $\Psi_Y^{1,1} \in \exp(I_{q} \CF)$. Let $(Y,q)$ denote the constant $\CF$-path corresponding to $Y$. By Lemma~\ref{lem:trivial.holonomy},
\[ \Hol(X_0,p_0) = \Hol( (Y, q) \odot (X_0,p_0) ) = \Hol(X_1, p_1) \]
\end{proof}
\section{Comparison to Crainic-Fernandes Integration}\label{section:crainicfernandes}
Our construction of the fundamental groupoid and the holonomy groupoid of a singular foliation bears a significant similarity to the construction of the integration of a Lie algebroid by Crainic and Fernandes~\cite{CrFeLie}.
In fact, the inspiration behind our construction of the fundamental groupoid came from the observation that the definition of algebroid homotopy used in the integration of Lie algebroids can be defined purely in terms of time-dependent sections, rather than paths. In this section we will make the analogy concrete by showing that the two constructions are equivalent when restricted to a single leaf.
\subsection{\texorpdfstring{$A$}{A}-homotopy}\label{sec:ahomotopy}
Let us recall the construction of the integration of a Lie algebroid via $A$-paths.
\begin{defi}
Let $\pi \colon A \to M$ be a Lie algebroid with anchor map $\rho \colon A \to TM$. An $A$-path is defined to be a smooth function $a \colon [0,1] \to A$ such that $\rho \circ a = \d{(\pi \circ a)} / \d t$.
\end{defi}
The set of all $A$-paths is denoted $\CP(A)$. It has a homotopy-like equivalence relation called $A$-homotopy\footnote{The definition of $A$-homotopy is slightly non-standard but one can see it is equivalent by examining the proof of Proposition~1.3 of \cite{CrFeLie}.}.
\begin{defi}
Let $a_0$ and $a_1$ be $A$-paths. A smooth map $a(t,s) \colon [0,1] \times [0,1] \to A$ is called an $A$-homotopy from $a_0$ to $a_1$ if the following properties hold:
\begin{itemize}
\item $\rho \circ a = \d (\pi \circ a) / \d {t}$
\item $a(t,i) = a_i(t)$ for $i = 0,1$
\item Any two parameter family of sections $\alpha(t,s)$ of $A$ which extends $a(t,s)$ has the property that the unique solution, $\beta(t,s)$, of the initial value problem:
\begin{equation}\label{eqn:ahomotopy}
\frac{\d{ \alpha}(t,s)} { \d{s}} - \frac{\d{ \beta}(t,s) }{ \d{t}} = [\alpha(t,s) , \beta(t,s)] \qquad \beta(0,s) = 0
\end{equation}
satisfies $\beta(1,s)|_{(\pi \circ a)(1,s)} = 0$.
\end{itemize}
\end{defi}
\begin{lem}\label{eqn:tangentahomotopy}
Suppose $\alpha(t,s)$ is a two-parameter family of vector fields on $M$ and let $\Phi_\alpha^{t,s} \colon M \to M$ denote the flow of $\alpha$ in the $t$-direction. Then the two parameter family of vector fields
\[\beta(t,s) := \frac{\d{ \Phi}^{t,s}_\alpha }{\d{s}} \]
satisfies Equation~(\ref{eqn:ahomotopy}).
\end{lem}
\begin{rem}
The proof of proceeding lemma is a direct calculation, which we omit. It is the main technical observation which shows that $TM$-homotopy coincides with classical endpoint preserving homotopies.
\end{rem}
\subsection{Leafwise Integration and \texorpdfstring{$\Mon(\CF)$}{Mon(F)}}
To see the relationship between Lie algebroids and foliations, we return to the topic of the fiberspace:
\[ \pi \colon A(\CF) \to M \]
from Definition~\ref{defi:fiberdiffeology}. In this section we will mildly abbreviate our notation for readability:
\[ A := A(\CF) \quad \text{ and } \quad A_L := \pi^{-1}(L) \]
For an arbitrary leaf $L$ of $M$, recall that $A_L$ is a Lie algebroid (Lemma~\ref{lem:ALisalgebroid}). Even better, we learned from Debord~\cite{Debord2013} that $A_L$ is always an integrable Lie algebroid.

Let $\Mon(\CF)_L$ denote the subgroupoid of $\Mon(\CF)$ of elements whose source is inside of a given leaf $L$. We would like to compare $\Mon(\CF)_L$ with $\mathcal{G}(A_L)$. Indeed, it turns out that they are the same.
\begin{thm}\label{thm:fundamentalgrpd}
The diffeological groupoids $\Mon(\CF)_L$ and $\mathcal{G}(A_L)$ are naturally diffeomorphic. In particular, $\Mon(\CF)_L$ is smooth.
\end{thm}
\begin{proof}
Let
\[ \CP(A) := \bigsqcup_{L \subseteq M} \CP(A_{L}) \subseteq C^\infty([0,1], A) \]
where the union is taken over all leaves $L$ of $\CF$. We equip $\CP(A)$ with the subset diffeology relative to $C^\infty([0,1], A)$. Two elements of $\CP(A)$ are said to be $A$-homotopic if they are $A_L$-homotopic for some leaf $L$.

Consider the smooth function
\[Q \colon \CP(\CF) \to \CP(A) \]
\[ Q(X, \gamma)(t) := \ev_{\gamma(t)} (X(t)) \]
The theorem follows from two claims:
\begin{enumerate}
\item $Q$ is a subduction.
\item Suppose $(X_0 , \gamma_0 )$ and $(X_1, \gamma_1)$ are $\CF$-paths. Then $(X_0 , \gamma_0 )$ is $\CF$-homotopic to $(X_1, \gamma_1)$ if and only if $Q(X_0 , \gamma_0 )$ is $A$-homotopic to $Q(X_1, \gamma_1)$.
\end{enumerate}
\emph{ Claim 1:} We need to show that plots on $\CP(A)$ can be lifted to plots on $\CP(\CF)$. Suppose $V$ is a subset of some euclidean space and $0 \in V$. Let $\phi \colon V \to \CP(A)$ be a plot. To prove the claim, we must show that there exists an open neighborhood $W$ of $0 \in V$ together with a lift $\til \phi \colon W \to \CP(\CF)$ such that $Q \circ \til \phi = \phi|_{V}$.

The smoothness of $\phi$ is equivalent to the smoothness of:
\[ a \colon V \times [0,1] \to A \qquad a(v,t) := \phi(v)(t) \]
Let
\[ \gamma := \pi \circ a \colon V \times [0,1] \to M \]
Since $[0,1]$ is compact, it follows that $\gamma(W \times [0,1])$ is contained inside of the interior of a compact subset of $M$. Therefore, we assume that $\CF(M)$ is finitely generated without loss of generality.

Let $\{ X_i \}_{i=1}^n \subseteq \CF(M)$ be a set of generators.
Using the coefficient diffeology on $\CF(M)$ and the compactness of the interval, we conclude that there exists an open neighborhood $W$ of $0 \in V$ together with smooth functions $\{ c^i \} \in C^\infty_{W \times [0,1]}$ such that for all $(w,t) \in W \times [0,1]$ we have that:
\[ a(w,t) = \sum \ev_{\pi \circ a(w,t)} (c^i (w,t) X_i) \]
Let $\til \phi(w) := (X_w, \gamma_w)$ where:
\[ X_w(t) := \sum c^i(w,t) X_i \qquad \gamma_w(t) := \pi \circ a(w,t) \]

\emph{ Claim 2:} $(\Rightarrow)$ Suppose $(X_0, \gamma_0)$ is $\CF$-homotopic to $(X_1, \gamma_1)$. Let $L$ be the unique leaf containing the image of $\gamma_1$ and $\gamma_2$. By assumption, there exists a variation $(X(t,s), \gamma(t,s))$ with complement $Y(t,s)$ such that $Y(1,s) \in I_{\gamma(1,s)} \CF$.
Lemma~\ref{eqn:tangentahomotopy} implies that:
\[ \frac{\d{X}(t,s)} { \d{s}} - \frac{\d{Y}(t,s) }{ \d{t}} = [X(t,s) , Y(t,s)] \quad \text{ and } \quad Y(0,s) = 0 \]
Let $\alpha(t,s)$ and $\beta(t,s)$ denote the two parameter sections of $A_L$ represented by $X$ and $Y$ respectively.
Since the algebroid bracket on $A_L$ comes from the bracket on $\CF$, it follows that $\alpha(t,s)$ and $\beta(t,s)$ satisfy Equation~\ref{eqn:ahomotopy}.
Furthermore, since $Y(1,s) \in I_{\gamma(1,s)} \CF$ it follows that $\beta(1,s) = 0$.
Therefore $\pi(X_0,\gamma_0)$ and $\pi(X_1, \gamma_1)$ are $A_L$-homotopic.

$(\Leftarrow)$ We will show the other direction in two parts. First, we will show that the fibers of $\pi$ are contained in the $\CF$-homotopy equivalence classes. Then we will show that an $A_L$-homotopy can be lifted to a $\CF$-homotopy between fibers.

Suppose $\pi(X_0,\gamma_0)=\pi(X_1, \gamma_1)$. In this case, $\gamma_0 = \gamma_1$. Let $X(t,s) := s X_1(t) + (1-s) X_0(t)$. We claim that $(X(t,s), \gamma)$ is a $\CF$-homotopy. Let $Y(t,s)$ denote the complement of this variation. Recall Equation~\ref{eqn:variation} for $Y$:
\[ Y(1,s) = (\Phi^{1,s}_X)_* \int\limits_{0}^1 {(\Phi^{u,s}_X)}_*^{-1} (X_1(u) - X_0(u)) \text{d} u \]
Since $(\Phi^{t,s}_X)_*$ preserves the foliation (Lemma~\ref{lem:preserves.foliation}), and since $(X_1(u) - X_0(u))|_{\gamma(u)} \in I_{\gamma(u)} \CF$. We can conclude that $Y(1,s) \in I_{\gamma(1)} \CF$. Which proves that $(X_0, \gamma)$ and $(X_1, \gamma)$ are homotopic.

Now we show that $A_L$ homotopies can be lifted to $\CF$-homotopies between fibers of $\pi$. Suppose $a(t,s)$ is an $A_L$ homotopy between some $A_L$-paths $a_0$ and $a_1$.
Let $\gamma(t,s)$ denote the base of this map. Using the argument from Claim 1, we can lift the map $a \colon [0,1]^2 \to A_L$ to a map $\alpha \colon [0,1]^2 \to \CF(M)$ (we can always lift the map locally and then patch together using a partition of unity). Let $\beta(t,s)$ be the complement of the associated variation $(\alpha, \gamma)$. Since we know $a(t,s)$ is an $A$-homotopy, it follows that $[\beta(1,s)]|_{A_\gamma(1,s)} = 0$ for all $s$. This implies that $\beta(1,s) \in I_{\gamma(1,s)} \CF$. Therefore $(X, \gamma)$ is a $\CF$-homotopy. Furthermore $\pi(X(t,0), \gamma(t,0)) = a_0$ and $\pi(X(t,1), \gamma(t,1)) = a_1$. This concludes Claim 2.
\end{proof}
This shows that our construction of $\Mon(\CF)$ is (leaf-wise) equivalent to performing a classical integration of $A$. An advantage of our construction of $\Mon(\CF)$ over constructing the integration separately over each leaf is the fact that $\Mon(\CF)$ inherits a diffeological structure which glues the leaf-wise integrations together.
\section{Comparison to the Androulidakis-Skandalis construction}\label{section:compareas}
A version of the holonomy groupoid originally appeared in a paper of Androulidakis and Skandalis~\cite{AndrSk}.
Their construction does not bear much resemblance to ours. The main result of this section (Theorem~\ref{thm:main}) is the fact that the construction in Section~\ref{subsec:holonomy} is equivalent (as a diffeological groupoid) to the Androulidakis-Skandalis construction. In order to prove this result, we will need to review the construction of Androulidakis and Skandalis.

It is in this section that we will most use the language of diffeology. In addition to providing us with a `smooth sense' in which our construction is the same as that of Androulidakis-Skandalis, the diffeology itself is ultimately crucial in our proof that the map we construct is a bijection. 
\subsection{Bisubmersions}
\begin{defi}
A \emph{bisubmersion} over $(M, \CF)$ is a manifold $P$ equipped with two submersions:
\[\begin{tikzcd}
V & P \arrow[swap]{l}{\mathbf{t}} \arrow{r}{\mathbf{s}"} & U
\end{tikzcd}\]
where $U$ and $V$ are open subsets of $M$. These maps must also be compatible with the foliation in the following sense:
\begin{equation}\label{eqn:bisubmersioncondition}
\mathbf{t}^{-1}(\CF) = \Gamma(\ker \d { \mathbf{t}}) + \Gamma(\ker \d{\mathbf{s}}) = \mathbf{s}^{-1}(\CF)
\end{equation}
A \emph{bisection} of $P$ is defined to be a smooth map $\sigma \colon U \to P$ such that $\mathbf{s} \circ \sigma = 1_U$ and $\mathbf{t} \circ \sigma \colon U \to V$ is a diffeomorphism.\end{defi}
A bisection $\sigma$ is said to be through a point $p \in P$ if $p$ is in the image of $\sigma$. In such a case, we say that $p \in P$ carries the diffeomorphism $\mathbf{t} \circ \sigma$.

In practice, we are mainly interested in \emph{local} bisections. That is, bisections for which the domain $U$ is small enough. Indeed, we can see that if we permit shrinking the open set $U$, one can always find a bisection through a given point $p \in P$.
\begin{ex}[Path Holonomy Bisubmersion]
Let $x \in M$ and suppose $\{ X_i \}^q_{i =1}$ is a set of compactly supported elements of $\CF$ which represent a basis for $A_x := \CF / I_x \CF$. Then let $\mathbf{t} \colon \mathbb{R}^q \times M \to M$ be the map defined as below:
\[ \mathbf{t} (c_1, \ldots , c_q , y) \mapsto \Phi^1_{\sum_{i=1}^q c_i X_i} (y) \]
and let $s \colon \mathbb{R^q} \times M \to M$ be the projection to the second component.
Then there exists an open neighborhood $P \subseteq \mathbb{R}^q \times M$ which contains $(0, \ldots ,0, x)$ such that:
\[
\begin{tikzcd} M \supset V & P \arrow[swap]{l}{\mathbf{t}} \arrow{r}{\mathbf{s}} & U \subseteq M \end{tikzcd}
\]
is a bisubmersion.
\end{ex}
It is possible to compose bisections via a fiber product operation:
\[
\begin{tikzcd}[row sep = small]
& &Q \times_{U_2} P \arrow{dl} \arrow{dr} & &\\
&Q \arrow{dr} \arrow{dl} & & P \arrow{dr} \arrow{dl} & \\
U_3 & & U_2 & & U_1
\end{tikzcd}
\]
Let $\PHB$ be the disjoint union of all compositions of path holonomy bisubmersions. The manifold $\PHB$ also has a monoidal structure since given $p \in P \subseteq \PHB$ and $q \in Q \subseteq \PHB$ such that $\mathbf{s}(p) = \mathbf{t}(q)$ we can define $p \cdot q := (p,q) \in P \times_M Q \subseteq \PHB$.
\begin{defi}
The holonomy groupoid of Androulidakis and Skandalis, denoted $\HolBis$ is defined to be $\PHB / \sim$ where $p \sim q$ if and only if they carry the same diffeomorphism. The source and target maps of the groupoid are inherited from those of $\PHB$.
\end{defi}
\newcommand{\AS}{\CH(\CF)^{AS}}
\newcommand{\GV}{\CH(\CF)^{GV}}
For the sake of clarity, we will denote $\CF$-paths up to holonomy by $\CH(\CF)^{GV}$ to distinguish it from $\CH(\CF)^{AS}$.

In section~\ref{subsec:holonomy}, we defined $\GV$ by associating holonomy transformations to $\CF$-paths. Before our work, Androulidakis and Zambon~\cite{AZ2} showed that it was possible to define holonomy transformations for $\CH(\CF)^{AS}$ as well. We now briefly summarize their results.
Suppose at each $x \in M$ we have chosen a slice $S_x$ and formed the corresponding groupoid of holonomy transformations $\HT$. Suppose $p \in \PHB$ is a point in a bisubmersion $P$.
Let $f \colon U \to V$ be a locally defined diffeomorphism carried by $p$ then \emph{holonomy} of $p$, written $\Hol^{AZ}(p)$ is defined to be the class of $f$ in $\HT$.
\begin{thm}[Androulidakis-Zambon]
The map $\Hol^{AZ} \colon \PHB \to \HT$ is well defined and the fibers of $\Hol^{AS}$ are precisely the fibers of the projection $\PHB \to \AS$.
\end{thm}
A consequence of this result is another characterization of the equivalence relation on $\PHB$ which leads to $\AS$.
\subsection{Relating \texorpdfstring{$\AS$}{AS} and \texorpdfstring{$\GV$}{GV}}

To understand the relationship between these two constructions, we begin by defining a map:
\[ \til \Psi \colon \PHB \to \CP(\CF). \]
Suppose $(c_1 , \ldots c_q, x)$ is a point in a path holonomy bisubmersion $P \subseteq \mathbb{R}^q \times M$ with associated basis $\{ X_i \}_{i =1}^q$. Let:
\[ \til \Psi(c_1, \ldots c_q, x) := \left( \sum_{i = 1}^q c_i X_i , \,\Phi^t_{\sum_{i=1}^q c_i X_i} (x)\right). \]
If $p = (p_1, p_2) \in \PHB$ is an element of a bisubmersion $P$ which comes as a product $P = P_1 \times_M P_2$. Then we define $\til \Psi$ inductively as $\til \Psi(p_1) \odot \til \Psi(p_2)$ using the product of $\CF$-paths (Definition~\ref{defi:concat}).

\begin{thm}\label{thm:main}
The smooth map $\til\Psi \colon \PHB \to \CP(\CF)$ descends to a diffeological groupoid isomorphism:
\[ \Psi \colon \AS \to \GV \]
\end{thm}
We will split the proof into a series of lemmas.
\begin{lem}
The map $\til \Psi \colon \PHB \to \CP(\CF)$ is compatible with holonomy transformations, i.e. it makes the following diagram commute:
\[
\begin{tikzcd}
\PHB \arrow{d}{\Hol^{AZ}} \arrow{r}{\til\Psi} & \CP(\CF) \arrow{d}{\Hol} \\
\HT \arrow[equal]{r} & \HT
\end{tikzcd}
\]
\end{lem}
\begin{proof}
First, observe that $\til \Psi$ is compatible with the monoidal structures on $\PHB$ and $\CP(\CF)$ by definition. Hence, we only need to show this diagram commutes for a set of elements which generate $\PHB$.

Let $p = (c_1, \ldots , c_q, x)$ be a point inside of a path holonomy bisubmersion $P \subseteq \mathbb{R}^q \times M$ associated to some set of local generators $\{ X_i \}_{i=1}^q$. There is a canonical bisection $\sigma(y) := (c_1, \ldots , c_q, y)$ which contains $p$. By the definition of $\bt \colon P \to M$ the diffeomorphism defined by $\sigma(y)$ around $x$ is precisely the flow $f := \Phi^1_{\sum_{i=1}^q c_i X_i}$ centered around $x$. Hence, the holonomy of $p$ is the class of $f$ in $\HT$.

However, $\til \Psi(p)$ is defined to be the $\CF$-path $(\sum_{i=1}^q c_i X_i , \gamma(t))$, where $\gamma$ is the integral curve of $\sum_{i=1}^q c_i X_i$ starting at $x$. Therefore the holonomy of $\til \Psi(p)$ is also represented by $f$.
\end{proof}
From this lemma, we can conclude the following:
\begin{cor}
The map $\til \Psi$ descends to an injective homomorphism of diffeological groupoids $\Psi \colon \AS \to \GV$.
\end{cor}
\begin{proof}
If $p_1$ and $p_2 \in \PHB$ represent the same element of $\AS$, then they must have the same holonomy\cite{AZ2}. By the previous lemma, $\til \Psi(p_1)$ must have the same holonomy of $\til\Psi(p_2)$. Since $\til\Psi$ respects the holonomy equivalence relation, it defines a map $\Psi \colon \AS \to \GV$ on equivalence classes. $\Psi$ is a homomorphism since $\til \Psi$ is compatible with the monoidal structures on $\PHB$ and $\CP(\CF)$ that induce the product on $\AS$ and $\GV$, respectively. Lastly, $\Psi$ is injective since it fits into a commuting diagram:
\[
\begin{tikzcd}
\AS \arrow[hookrightarrow]{d} \arrow{r}{\Psi} & \GV \arrow[hookrightarrow]{d} \\
\HT \arrow[equal]{r} & \HT
\end{tikzcd}
\]
\end{proof}
Now that we know that $\Psi$ is injective. Our next goal is to show that $\Psi$ is a \emph{local subduction}~\cite{Diffeology}. More concretely, plots centered on points in the image of $\Psi$ locally factor through $\AS$. For that, we will need a lemma about lifting vector fields to bisubmersions.
\begin{lem}
Suppose \(\begin{tikzcd}[column sep=small] U_2 & P \arrow[swap]{l}{\bt} \arrow{r}{\bs} & U_1 \end{tikzcd}\) is a bisubmersion. Given a smooth function $X \colon V \times [0,1] \to \CF(U_2)$ for $V \subseteq \mathbb{R}^n$ such that $X(0,t) = 0$, there exists a smooth lift
\[ \til X \colon V \times [0,1] \to \Gamma(\ker \d{\bs}) \]
such that $\bt_* \til X = X$ and $\til X(0,t) = 0$.
\end{lem}
\begin{proof}
From the definition of the diffeology on $\CF(U_2)$, we know that $X$ can be written in the form $X(a_1, \ldots, a_n, t) = \sum_{i=1}^q c_i(a_1, \ldots a_n,t) Y_i$ for some constant elements $Y_i \in \CF(U_2)$ and smooth functions $c_i \colon V \times [0,1] \to \mathbb{R}$. We can assume that $c_i(0,t) = 0$ for all $i$ since otherwise we can just replace each $c_i$ with $c_i - c_i(0,t)$ if necessary.

From the bisubmersion condition we know that $\bt^{-1}(\CF) = \Gamma(\ker \bt) + \Gamma(\ker \bs)$. Since $\bt^{-1}(\CF)$ is defined to be $(\d{\bt})^{-1}( \bt^* \CF(U_2))$, we know that $\d{\bt} \colon \Gamma(\ker \bs) \to \bt^* \CF(U_2)$ is a surjective morphism of $C^{\infty}_P$-modules. Therefore, we can choose $\til Y_i \in \Gamma(\ker \bs)$ such that $\bt_* \til Y_i = Y_i$.

Let $\til X(a_1 , \ldots, a_n, t) := \sum_{i=1}^n c_i(a_1, \ldots a_n,t) \til Y_i$. From the definition of $\til Y_i$, we obtain that $\til X$ satisfies the claim of the lemma.
\end{proof}
The main application of this lemma is the following proposition which says, roughly, that plots centered on identity elements of $\GV$ locally factor through $\Psi$.
\begin{prop}\label{prop:subductionatid}
Let $x \in M$. Suppose $\phi \colon V \to \GV$ is a plot such that $0 \in V$ and $\phi(0) = 1_x \in \GV$. Then there exists an open neighborhood of the origin, $V' \subseteq V$, and a plot $\psi \colon V' \to \AS$ such that $\Psi \circ \psi = \phi|_{V'}$.
\end{prop}
\begin{proof}
Suppose $\phi \colon V \to \GV$ is a plot as in the statement of the lemma. Let $W \subseteq V$ be a smaller neighborhood of the origin such that we can lift $\phi|_W$ to a plot $\til \phi \colon W \to \CP(\CF)$.
Let us write the components of $\til\phi$ in the following manner:
\[ \til\phi (a_1 , \ldots , a_n)(t) = \left(X(a_1, \ldots a_n , t), \gamma(a_1, \ldots, a_n ,t) \right) \]

Now suppose \(\begin{tikzcd}[column sep=small] U_2 & P \arrow[swap]{l}{\bt} \arrow{r}{\bs} & U_1 \end{tikzcd}\) is a path holonomy bisubmersion such that the image of $\gamma$ is contained in $U_2$. Such a path holonomy bisubmersion exists since we can make the image of $\gamma$ as small as we like by shrinking $W$. Furthermore, we can assume that the support of $X(a_1, \ldots , a_n, t)$ is uniformly contained in $U_2$.

The previous lemma implies that there exists $\til X \colon W \times [0,1] \to \ker \d{\bs} \le \mathfrak{X}(P)$ such that $\bt_* \til X = X|_{U_2}$ and such that $\til X(0,t) = 0$. By possibly shrinking $W$, further, we can ensure that the flow $\Phi^1_{\til X}$ is well defined in some neighborhood of $ \{ 0 \} \times U_1 \subseteq P$.

Let $\til \psi \colon W \to P \subseteq \PHB$ be defined by:
\[ \til \psi(a_1, \ldots, a_n) = \Phi^1_{\til X(a_1, \ldots , a_n, \cdot)} \left(0, \ldots , 0 , \gamma(a_1, \ldots a_n, 0) \right) \]
Such a function defines a plot $\psi \colon W \to \AS$. We claim that $\Psi \circ \psi = \phi|_W$. To show this, we will show that $\Hol \til \psi = \Hol \til\phi$.

For each $(a_1 , \ldots a_n)$, let $\sigma(a_1, \ldots , a_n) \colon U_1 \to P$ be the bisection defined by
\[ \sigma(a_1 , \ldots , a_n) := \Phi^1_{\til X(a_1 , \ldots , a_n , \cdot )} \sigma^{id} \]
Where $\sigma^{id}$ is the canonical bisection of $P$. From the definition of $\til X$, it follows that the diffeomorphism carried by $\sigma(a_1, \ldots , a_n)$ is precisely the flow of the time dependent vector field $X|_{\{a_1 \} \times \ldots \times \{a_n\} \times [0,1]}$. Furthermore, $\sigma(a_1 , \ldots, a_n)$ is a bisection through $\til \psi(a_1 , \ldots, a_n)$. This shows that the holonomy of $\til \psi(a_1 , \ldots , a_n)$ must be equal to the holonomy of $\til \phi(a_1 , \ldots a_n)$ which completes the proof.
\end{proof}
We can use the special case of a plot centered on an identity, to show that the preceding proposition holds for any plot which is centered around an element in the image of $\Psi$.
\begin{cor}\label{cor:plots.factor}
Let $x \in M$. Suppose $\phi \colon V \to \GV$ is a plot such that $0 \in V$ and $\phi(0) \in \Psi(\AS) \subseteq \GV$. Then there exists an open neighborhood of the origin, $V' \subseteq V$, and a plot $\psi \colon V' \to \AS$ such that $\Psi \circ \psi = \phi|_{V'}$.
\end{cor}
\begin{proof}
Suppose $\phi \colon V \to \GV$ is a plot such that $\phi(0) = h \in \Psi(\AS) \subseteq \GV$. Let $g \in \AS$ be an element such that $\Psi(g) = h$. Furthermore, let $\sigma$ be a (locally defined) bisection of $\AS$ through $g$. Shrink $V$ if necessary so that $\sigma( \bt \circ \phi)$ is well defined.

Then $\overline{\phi} := \Psi \circ \sigma( \bs \circ \phi)^{-1} \cdot \phi$ is a well defined plot centered on a unit element of $\GV$. Let $\overline{\psi} := V \to \AS$ be a plot such that $\Psi \circ \overline{\psi} = \overline{\phi}$. Such a plot exists by Proposition~\ref{prop:subductionatid}. Let $\psi := \sigma(\bs \circ \phi) \cdot \overline{\psi}$. Then:
\begin{align*}
\Psi \circ \psi &= \Psi \circ (\sigma(\bs \circ \phi) \cdot \overline \psi ) \\
&= (\Psi \circ \sigma(\bs \circ \phi)) \cdot (\Psi \circ \overline \psi) \\
&= (\Psi \circ \sigma(\bs \circ \phi)) \cdot (\Psi \circ \sigma(\bs \circ \phi)^{-1}) \cdot \phi \\
&= \phi
\end{align*}
which concludes the proof.
\end{proof}
We can now complete the proof of Theorem~\ref{thm:main}.
\begin{proof}(of Theorem~\ref{thm:main})
Let us now equip $\GV$ and $\AS$ with their respective D-topologies (II.8 in \cite{Diffeology}). That is, a subset of $\GV$ or $\AS$ is open if and only if its inverse image under every plot is open. Since $\Psi$ is smooth, it follows that it is continuous with respect to the D-topology. Furthermore, Corollary~\ref{cor:plots.factor} showed us that plots on $\GV$ which intersect elements in the image of $\Psi$, locally factor through $\Psi$. By II.20 from \cite{Diffeology}, this implies that $\Psi$ is open.

These topologies make both $\GV$ and $\AS$ source connected topological groupoids. Since $\Psi$ is open, we have that $\Psi$ is surjective~\cite{MK2}. This implies that all plots on $\GV$ locally factor through $\AS$. Since $\Psi$ is also injective, we conclude that it is a diffeomorphism of diffeological spaces.
\end{proof}

\section{Functoriality}\label{section:functoriality}
Our aim in this section is to understand the functorial properties of the construction of $\Mon(\CF)$ and $\CH(\CF)$. The main technical difficulty in this section is that it is not generally possible to push forward vector fields. A consequence of this is that it is usually difficult to define a function at the level of $\CF$-paths. To circumvent this problem, we will require some technical development regarding the relationship between a foliation $\CF$ and its fiberspace $A(\CF)$.
\subsection{Comorphisms of sheaves of modules}
We need to develop a notion of morphism of foliated manifolds. Our model example is the case of projective foliations where a morphism of foliations will corresponds to a morphism of the associated Lie algebroids. Higgins and Mackenzie~\cite{hm1993} observed that there is a duality between morphisms of vector bundles and comorphisms of modules. This observation will motivate our definition of a morphism of foliated manifolds.
\begin{defi}
Suppose $\CF$ is a sheaf of $C^\infty_N$-modules on $N$ and $f \colon M \to N$ is a smooth function. The \emph{ pre-pullback $\overline{ f^! \CF}$ of $\CF$ along $f$} is a presheaf of $\Cinf$-modules such that for all open $U \subseteq M$:
\[ \overline{f^! \CF} (U) := \Cinf(U) \otimes_{C^\infty_N(N)} \CF(N) \]
Given $u \in C^\infty_N(N)$, the action on $C^\infty_M(U)$ is by the usual pullback operation:
\[ v \mapsto v \cdot (u \circ f)|_{U} \qquad v \in C^\infty_M(U)\]
The \emph{ pullback $f^! \CF$ of $\CF$ along $f$} is defined to be the sheafification of $\overline{f^! \CF}$.
\end{defi}
\begin{rem}
The pre-pullback of $\CF$ is a separated presheaf. In particular, the canonical sheafification map $\overline{f^! \CF} \to f^! \CF$ is injective and locally an isomorphism. Therefore, for any point $p \in M$ there exists an open neighborhood $U$ such that:
\[ f^! \CF (U) = \Cinf(U) \otimes_{C^\infty_N} \CF(N) \]
\end{rem}
If $\CF$ is locally finitely generated, then it follows that $f^! \CF$ is locally finitely generated. Generally, it is easier to work with elements in the image of the inclusion $\overline{f^! \CF} \into f^! \CF$ since they have an explicit form in terms of the tensor product.

The pullback we have just defined is a pullback in the world of modules, not singular foliations. In general, the pullback of a singular foliation (as a sheaf of modules) is only a module and not a singular foliation. It is a distinct operation from the inverse image $f^{-1}(\CF)$ and the pullback $f^* \CF$ operations seen in Androulidakis and Skandalis~\cite{AndrSk}.
\begin{ex}\label{ex:projectivepullback}
Suppose $E \to N$ is a vector bundle over $N$ and $f \colon M \to N$ is smooth. Let $\Gamma_E$ denote the associated sheaf of $C^\infty_N$-modules and $f^! E := M \times_N E$ denote the pullback vector bundle. Given an element:
\[ \sum u^i \otimes s_i \in \Cinf(U) \otimes_{C^\infty_N} \Gamma_E(N) \]
one can define an element $\sigma$ of $\Gamma_{f^!E}(U)$:
\[ \sigma(p) = \left( p, \,\sum u^i(p) (s_i \circ f)(p) \right) \]
This correspondence gives rise to a morphism of presheaves $\overline{f^! \Gamma_E} \to \Gamma_{f^! E}$. Since it is locally an isomorphism, the associated morphism of sheaves $f^! \Gamma_E \to \Gamma_{f^!E}$ is an isomorphism.
\end{ex}
The pullback is a functorial operation. That is, given a morphism $G \colon \CF_1 \to \CF_2$ of $C^\infty_N$-modules, then there is a canonical morphism of $\Cinf$-modules:
\[ f^! G \colon f^! \CF_1 \to f^!\CF_2 \]
which is uniquely determined by its behavior on the image of the pre-pullback:
\[( f^! G) \left(\sum u^i \otimes s_i \right) = \sum u^i \otimes G(s_i) \]
\begin{defi}
Suppose $\CF_M$ and $\CF_N$ are sheaves of modules on manifolds $M$ and $N$. A \emph{ comorphism}
\[ (F,f) \colon\CF_M \to \CF_N \]
consists of a pair $(F , f)$ where $f \colon M \to N$ is a smooth map and $F \colon \CF_M \to f^! \CF_N$ is a morphism of sheaves of $\Cinf$-modules.

Comorphisms are composed according to the following rule:
\[ (F, f) \circ (G, g) := (F \circ f^!G , f \circ g) \]
\end{defi}
\begin{ex}\label{ex:projectivecomorphism}
Let $E \to M$ and $W \to N$ be vector bundles over $M$ and $N$ respectively. Suppose $\Psi \colon E \to W$ is a vector bundle morphism which covers $f \colon M \to N$. Let
\[ \Psi_* \colon \Gamma_E \to \Gamma_{f^! W} \]
be the associated pushforward map. If we identify $\Gamma_{f^! W}$ with $f^! \Gamma_{W}$ using the canonical isomorphism found in Example~\ref{ex:projectivepullback} then $(\Psi_* , f) \colon \Gamma_E \to \Gamma_W$ is a comorphism.
\end{ex}
In fact, it turns out that all comorphisms $\Gamma_E \to \Gamma_W$ arises in this way. We will see how to construct the bundle map associated to a comorphism in the next section, where we will do it in a more general setting. The next example is a special case of the one just discussed.
\begin{ex}
Suppose $f \colon M \to N$ is a smooth map. Then $f^! \SX_N$ is canonically isomorphic to the sheaf of sections of $f^! TN$. The differential of $f$ defines a morphism of vector bundles $TM \to f^! TN$ and hence a comorphism:
\[ (\d f_*, f) \colon \SX_M \to \SX_N \]
\end{ex}
\subsection{Comorphisms and the fiberspace}
Higgins and Mackenzie~\cite{hm1993} showed that there is a functor from projective locally finitely generated modules with comorphisms as maps to the category of vector bundles. In this section, we will show that one can extend this functor to arbitrary modules by allowing for diffeological vector bundles.

The motivation for doing this is that working at the level of $A(\CF)$ will permit us to circumvent the problem of pushing forward vector fields. Although one cannot push forward $\CF$-paths, it makes sense to push forward a path in the fiberspace $A(\CF)$.

Let $\CF$ be a locally finitely generated $\Cinf$-module. Recall the definition of the fiberspace of $\CF$ (Definition~\ref{defi:localization}):
\[ A(\CF) := \bigsqcup_{p \in M} A(\CF)_p \qquad A(\CF)_p := \frac{\CF(M)}{I_p \CF(M)} \]
Recall from Definition~\ref{defi:fiberdiffeology} that we equip the fiberspace with a quotient diffeology via the evaluation map:
\[ \ev \colon M \times \CF(M) \to A(\CF) \qquad (p, X) \mapsto \ev_p (X) \]
The addition, multiplication and zero section operations on $M \times \CF(M)$ descend to smooth functions on $A(\CF)$ which make it into a diffeological vector bundle over $M$.

We begin with a lemma that tells us that the fiberspace construction plays well with pullbacks of modules.
\begin{lem}\label{lemma:diffeologicalpullback}
Suppose $\CF$ is a $C^\infty_N$-module and $f \colon M \to N$ is a smooth map. Then there is a morphism:
\[ q_f \colon A(f^! \CF) \to A(\CF) \]
of diffeological vector bundles covering $f \colon M \to N$.
\end{lem}
\begin{proof}
Observe that for all $p \in M$, there is a $C^\infty_N$-bilinear function:
\[ C^\infty_M(M) \times \CF_N(N) \to A(\CF_N)_{f(p)} \qquad (u, X) \mapsto u(p) \cdot \ev_{f(p)} (X) \]
This induces a smooth function:
\[ Q_f \colon M \times (f^! \CF)(M) \to A(\CF) \]
which is uniquely determined by its behavior on elements in the pre-pullback:
\[ Q_f \left( p, \sum u^i \otimes X_i \right) := \sum u^i(p) \cdot \ev_p (X_i) \]
Since $Q_f(p, Y) = 0$ for all $Y \in I_p (f^! \CF_N)$, we get that $Q_f$ descends to a diffeological vector bundle map:
\[ q_f \colon A(f^! \CF) \to A(\CF) \]
\end{proof}
\begin{ex}
If $\CF$ is the sheaf of sections of a vector bundle $E \to N$. Then we have a natural identification:
\[ A(f^! \CF) = f^! E := M \times_N E \]
Under this identification, $q_f \colon f^! E \to E$ is projection to the second component.
\end{ex}
\begin{defi}
Suppose $(F, f) \colon \CF_M \to \CF_N$ is a comorphism of modules. Let:
\[ \overline{F} \colon A(\CF_M) \to A(f^!\CF_N) \qquad \ev_p(X) \mapsto \ev_p(F(X)) \]
Then we define:
\[ A(F,f) := q_f \circ \overline{F} \colon A(\CF_M) \to A(f^! \CF_N) \to A(\CF_N) \]
where $q_f$ is as in Lemma~\ref{lemma:diffeologicalpullback}.
\end{defi}
\begin{ex}
Suppose $\CF_M = \Gamma_E$ and $\CF_N = \Gamma_W$ are sheaves of sections of a vector bundle. Then $A(\Gamma_E) \to M$ is canonically isomorphic to $E \to M$ and $A(\Gamma_W) \to N$ is canonically isomorphic to $W \to N$. Using these identifications, we get vector bundle morphism
\[ A(F, f) \colon E \to W \]
Furthermore, if we apply the reverse construction seen in Example~\ref{ex:projectivepullback}, we will recover the comorphism $(F,f)$.
\end{ex}
It is straightforward to check that the construction of the fiberspace $A(\CF)$ and the diffeological vector bundle map $A(F,f)$ yields a functor $A$ from the category of manifolds equipped with sheaves of modules to the category of diffeological vector bundles over manifolds.
\subsection{Functoriality of \texorpdfstring{$\Mon(\CF)$}{Mon(F)}}
We now specialize our discussion to foliations rather than general modules. In this section, $\iota_M \colon \CF_M \into \SX_M$ is a foliation on $M$ and $\iota_N \colon \CF_N \into \SX_N$ is a foliation on $N$.
\begin{defi}
A comorphism $(F,f) \colon \CF_M \to \CF_N$ is a \emph{ morphism of foliated manifolds} if it satisfies the following properties:
\begin{itemize}
\item $F$ is compatible with the inclusions. That is, the following diagram commutes:
\begin{equation}\label{eqn:anchorcompatibility}
\begin{tikzcd}
\CF_M \arrow{r}{{F}} \arrow[hook]{d}{{\iota_M}} & f^* \CF_N \arrow{d}{{f^* \iota}} \\
\SX_M \arrow{r}{{\d f_*}} & f^* \SX_N
\end{tikzcd}
\end{equation}
\item $F$ is compatible with the brackets. That is, given $X_1, X_2 \in \CF_M(M)$ such that:
\[ F(X_1) = \sum_{i=1} u^i \otimes Y_i \qquad F(X_2) = \sum_{j=1} v^j \otimes Y_j\]
we have that:
\begin{equation}\label{eqn:bracket.compat}
F([X_1, X_2] ) = \sum_{i,j} u^i v^j \otimes [Y_i , Y_j] + \sum_j X_1(v^j) \otimes Y_j + \sum_i X_2(u^i) \otimes Y_i
\end{equation}
\end{itemize}
\end{defi}
\begin{ex}
Suppose $\CF_M$ and $\CF_N$ are projective foliations. Then a comorphism $(F,f) \colon \CF_M \to \CF_N$ is a morphism of foliated manifolds if and only if the associated vector bundle homomorphism:
\[ A(F,f) \colon A(\CF_M) \to A(\CF_N) \]
is a morphism of Lie algebroids. In this setting, Diagram~(\ref{eqn:anchorcompatibility}) is equivalent to compatibility with the anchor maps and Equation~(\ref{eqn:bracket.compat}) is equivalent to the usual compatibility with the Lie brackets (e.g. compare to Equation~(2.18) in \cite{LecturesIntegrabilty}.)
\end{ex}
\begin{ex}\label{ex:submoduleinclusion}
Suppose $\CF_1$ and $\CF_2$ are two foliations on $M$. Then a morphism of foliated manifolds $(F, \text{Id}_M) \colon \CF_1 \to \CF_2$ exists if and only if $\CF_1$ is a submodule of $\CF_2$. If such a morphism exists, it is unique since Diagram~\ref{eqn:anchorcompatibility} implies that $F \colon \CF_1 \to \CF_2$ is the submodule inclusion map.
\end{ex}
\begin{ex}\label{ex:morphismsubmersion}
Suppose $f$ is a submersion and $\CF_M = f^{-1}(\CF_N)$. That is, $\CF_M$ is the foliation generated by projectable vector fields $X \in \SX_M$ such that $\d f(X) \in \CF_N$. By choosing local sections of $f$, we can conclude that:
\[ f^! \iota_N \colon f^! \CF_N \to f^! \SX_N \]
is injective. Therefore, if we are given a morphism of foliated manifolds
\[ (F,f ) \colon \CF_M = f^{-1}(\CF_N) \to \CF_N \]
then Diagram~(\ref{eqn:anchorcompatibility}) implies that $F \colon \CF_M \to f^!\CF_N$ is just $\d f_*|_{\CF_M}$. \end{ex}
\begin{ex}\label{ex:factorization}
Suppose $(F, f) \colon \CF_M = \CF_N$ is a morphism of foliated manifolds and $f \colon M \to N$ is a submersion. Then Diagram~(\ref{eqn:anchorcompatibility}) implies that $\CF_M$ is a submodule of $f^{-1}(\CF_N)$. Therefore, $(F,f)$ factors uniquely into a composition
\[
\begin{tikzcd}[column sep = large]
& f^{-1} (\CF_N) \arrow{dr}{{(F_2, f)}} & \\
\CF_M \arrow{rr}{ {(F,f)}} \arrow{ur}{{ (F_1, \text{Id}_M) }} & & \CF_N
\end{tikzcd}
\]
That is $(F,f)$ is a composition of Example~\ref{ex:morphismsubmersion} with Example~\ref{ex:submoduleinclusion}. Consequently, when $f$ is a submersion, a morphism of foliated manifolds $\CF_M \to \CF_N$ exists if and only if $\CF_M \into f^{-1}(\CF_N)$. Furthermore, if such a morphism exists then it is unique.
\end{ex}
The bracket condition in the definition of morphism of foliated manifolds is not implied by compatiblity with the inclusion and a morphism of foliated manifolds $(F,f)$ is not always uniquely determined by $f$. The following two examples demonstrate these facts.
\begin{ex}
Let \(  M = \RR^2 \) with \( \CF_M = \SX_M \) and take \( N = \RR^2 \) with the foliation:
\[ \CF_N = \langle x \de_x , y \de_y , x \de_y , y \de_x \rangle \]
Now take \( f \colon M \to N \) to be the zero map \( f(x,y) = 0 \) and let \( F \colon \CF_M \to f^! \CF_N \) be such that:
\[ F(\de_x) = 1 \otimes x \de_y  \qquad F(\de_y) = 1 \otimes y \de_x \]
Then this pair \( (F,f) \) defines a comorphism of modules and satisfies compatibility with the inclusion. However, it does not satisfy compatibility with the bracket.
\end{ex}
\begin{ex}
Let $M=\RR$ with the full foliation and $N=\RR^2$ with the foliation given by the free module,
\[ \CF_N:= \langle R \rangle_{\CI_N} \]
with the inclusion
\[ \iota_N \colon \CF_N \to \SX_N \qquad R \mapsto y\de_x- x\de_y \]
Now let $f \colon \RR \to \RR^2$ be the constant map which sends every point to zero. Notice that the pullback of the inclusion $f^! \iota_N \colon f^!\CF_N \to f^! \SX_N$ is not injective. For example, $1 \otimes R$ is a non-trivial element of $f^!\CF_N$ but:
\[ (f^! \iota_N) (1 \otimes R) = 1 \otimes i_N(R) = 1 \otimes (y \de_x - x \de_y) = (y \circ f) \otimes \de_x - (x \circ f) \otimes \de_y = 0 \]
Indeed, given an arbitrary constant $C \in \RR$, we can define a module homomorphism:
\[ F^C \colon \SX_\RR \to f^* \CF_N \qquad F^C (\de_t) = C \otimes R \]
For each choice of $C$, $(F^C, f)$ is a morphism of foliated manifolds. Therefore, we have exhibited a family of distinct morphisms with the same base map.
\end{ex}
Suppose $(F, f)\colon \CF_M \to \CF_N$ is a morphism of foliated manifolds. Then Diagram~(\ref{eqn:anchorcompatibility}) implies that $f$ preserves the characteristic distribution associated to each foliation. In particular, for each leaf $L \into M$, there is a unique leaf $f_* L \into N$ containing $f(L)$. From now on, we will use this push-forward notation to refer to the induced function at the level of leaf-spaces.
\begin{thm}\label{thm:pi1isfunctorial}
Suppose we are given a morphism $(F, f) \colon \CF_M \to \CF_N$ of foliated manifolds. Then there exists a unique morphism of diffeological groupoids
\[ \Mon(F,f) \colon \Mon(\CF_M) \to \Mon(\CF_N) \]
such that the restriction to each leaf
\[ \Mon(F,f)_L \colon \Mon(\CF_M)_L \to \Mon(\CF_N)_{f_* L} \]
integrates the Lie algebroid morphism
\[A(F,f)_L \colon A(\CF_M)_L \to A(\CF_N)_{f_* L} \]
\end{thm}
\begin{proof}
The uniqueness follows from the fact that the source fibers of $\Mon(\CF_M)$ are simply connected and Lie's theorems for Lie algebroids. We only need to concern ourselves with constructing $\Mon(F , f)$.

Recall the set of $A(\CF_M)$-paths discussed in the proof of Theorem~\ref{thm:fundamentalgrpd}.
\[ \CP(A(\CF_M)) := \bigsqcup_{L} \CP(A(\CF_M)_L) \subseteq C^{\infty}([0,1],A(\CF_M)) \]
There is a similar set of $A(\CF_N)$-paths, $\CP(A(\CF_N))$. By post-composition of $A(F, f)$, we obtain a smooth function:
\[ \overline \Phi \colon \CP(A(\CF_M)) \to \CP(A(\CF_N)) \]
The definition of morphism of foliated manifolds implies that $A(F, f)$ is a Lie algebroid morphism when restricted to each leaf. Therefore $A(F, f)$ respects the algebroid homotopy equivalence relation.

By Claim 1 and Claim 2 in the proof of Theorem~\ref{thm:fundamentalgrpd}, $\overline \Phi$ descends to a smooth function:
\[ \Phi \colon \Mon(\CF_M) \to \Mon(\CF_N) \]
On each leaf, $\Phi$ has been constructed in exactly the same manner as the classical proof of Lie's second theorem for groupoids. Therefore, we take $\Mon(F, f) = \Phi$.
\end{proof}

\begin{ex}\label{ex:ME}
Suppose $f$ is a surjective submersion and $\CF_M = f^{-1}(\CF_N)$. We call an $\CF_M$-path $(X,\gamma)$ projectable if:
\[ X(t) = \sum c^i(t) X_i \]
where each $c^i(t) \colon [0,1] \to \mathbb{R}$ is a smooth function and $X_i \in \CF_M$ is a projectable vector field along $f$. Let $\CP_f(\CF_M) \subseteq \CP(\CF_M)$ denote the set of all $\CF_M$-paths which are projectable along $f$.

Since $\CF_M$ is generated by projectable vector fields, it follows that any $\CF_M$-path $Y$ is $\CF_M$-homotopic to a projectable $\CF_M$-path $X_Y$. Consequently, the map $\Phi$ from Theorem~\ref{thm:pi1isfunctorial} arises from a function at the level of projectable $\CF_M$-paths:
\[ \CP_f(\CF_M) \to \CP(\CF_N) \qquad (X,\gamma) \mapsto (\d f (X), \gamma) \]
\end{ex}
The construction of the homomorphism in Theorem~\ref{thm:pi1isfunctorial} yields a functor from the category of foliated manifolds to the category of diffeological groupoids. The compatibility of $\Mon$ with composition of morphisms follows from the usual functoriality of the construction of the universal integration of a Lie algebroid.

We warn the reader that $\Mon(F, f)$ does not always descend to a homomorphism at the level of the associated holonomy groupoids. By this we mean that there does not always exist a morphism of diffeological groupoids $\CH(\CF_M) \to \CH(\CF_N)$ which completes the following diagram:
\begin{equation}\label{diagram:completeit}
\begin{tikzcd}[column sep = large]
\Mon(\CF_M) \arrow{d} \arrow{r}{{\Pi_1(F,f)}} & \Mon(\CF_N) \arrow{d} \\
\CH(\CF_M) & \CH(\CF_N)
\end{tikzcd}
\end{equation}
This is already known to occur in the case of regular foliations. For example:
\begin{ex}
Suppose $N \to S^1$ is a non-trivial line bundle equipped with a flat connection (i.e. the Möbius band). We think of $N$ as a foliated manifolds by taking $\CF_N$ to be the vector fields tangent to the connection. Let $M = S^1$ be equipped with the full foliation. Both foliations are regular and the inclusion $f \colon M \into N$ of the zero section preserves the distributions so it induces a morphism of foliated manifolds $(F , f) \colon \CF_M \to \CF_N$.

The isotropy groups of $\Mon(\CF_M)$ are $\mathbb{Z}$ and the induced map $\Mon(\CF_M) \to \Mon(\CF_N)$ is the inclusion of the full subgroupoid at the zero section.

The isotropy groups of the holonomy groupoid $\CH(\CF_M)$ are trivial whereas the isotropy groups of $\CH(\CF_N)$ are $\mathbb{Z}_2$ on the zero section. Since there is no way to complete the following diagram of group homomorphisms
\[
\begin{tikzcd}[column sep = large]
\mathbb{Z} \arrow{d} \arrow{r}{1_{\mathbb{Z}}} & \mathbb{Z} \arrow{d} \\
1 & \mathbb{Z}_2
\end{tikzcd}
\]
there cannot exist a groupoid homomorphism $\CH(\CF_M) \to \CH(\CF_N)$ which completes Diagram~(\ref{diagram:completeit}).
\end{ex}
\begin{ex}
Suppose $M = \mathbb{R}^2$. Let $R \in \SX_M(M)$ be the usual counterclockwise rotation vector field. Now let $N := \mathbb{R}^3$ and let $X \in \SX_N(N)$ be the vector field:
\[ X := R \oplus -z \frac{\partial}{\partial z} \]
We take $\CF_M$ and $\CF_N$ to be the foliations generated by $R$ and $X$, respectively. Now let $f$ be the inclusion:
\[ f \colon M \to N \qquad (x,y) \mapsto (x,y,0) \]
Let $F \colon \CF_M \to f^* \CF_N$ be such that $F(R) = 1 \otimes X$. Then $(F, f)$ is a morphism of foliated manifolds. The smooth function $\Mon(\CF_M) \into \Mon(\CF_N)$ is the inclusion of the full subgroupoid over the $z = 0$ plane. In particular, it is an isomorphism at the level of isotropy groups.

If we compute the isotropy group of the holonomy of $\CF_M$ at the origin, we can see that it is isomorphic to $S^1$. On the other hand, the isotropy group of $\CH(\CF_N)$ at the origin is $\mathbb{R}$. Since there does not exist a non-trivial homomorphism of Lie groups $S^1 \to \mathbb{R}$, we conclude that $(F, f)$ does not induce a morphism $\CH(\CF_M) \to \CH(\CF_N)$.
\end{ex}
We conclude with a theorem that says that we can indeed complete Diagram~(\ref{diagram:completeit}) when $f$ is a submersion.
\begin{thm}\label{thm:holfunc}
Suppose $(F, f) \colon \CF_M \to \CF_N$ is a morphism of foliated manifolds and $f \colon M \to N$ is a submersion. There exists a diffeological groupoid homomorphism
\[ \Hol(F,f) \colon \CH(\CF_M) \to \CH(\CF_N) \]
which makes the following diagram commute:
\[
\begin{tikzcd}[column sep = large]
\Mon(\CF_M) \arrow{d} \arrow{r}{ {\Mon(F,f)}} & \Mon(\CF_N) \arrow{d} \\
\CH(\CF_M) \arrow{r}{ {\Hol(F,f)}} & \CH(\CF_N)
\end{tikzcd}
\]
\end{thm}
\begin{proof}
Let
\[ q_M \colon \Mon(\CF_M) \to \CH(\CF_M) \qquad q_N \colon \Mon(\CF_N) \to \CH(\CF_M) \]
be the natural quotient maps. We need to show that for all $g \in \Mon(\CF_M)$ such that $q_M(g)$ is an identity, we have that $q_N \circ \Mon(F,f)(g)$ is an identity element.

Recall the pullback foliation $f^{-1}(\CF_N)$ which is generated by $\Cinf$-linear combinations of projectable vector fields $X$ such that $\d f(X) \in \CF_N$. Since $f$ is a submersion, Example~\ref{ex:factorization} told us that $(F,f)$ factors uniquely through $f^{-1}(\CF_N)$. That is, there exist comorphisms:
\[ (F_1 , f) \colon f^{-1}(\CF_N) \to \CF_N \qquad (F_2 , \text{Id}_M) \colon \CF_M \to f^{-1}(\CF_N) \]
such that
\[(F, f) = (F_1 , f) \circ (F_2 , \text{Id}_M) \]
Therefore, we can split the proof into two cases. One case is where $f = \text{Id}_M$ and one case where $\CF_M = f^{-1}(\CF_N)$.

Suppose $f = \text{Id}_M$. Then the condition of being compatible with the inclusion implies that $F \colon \CF_M \to \CF_N$ is just the inclusion of a submodule of vector fields. In particular, $\Mon(F,f)$ sends the $\CF_M$ homotopy class of $(X,\gamma)$ to the $\CF_N$ homotopy class of $(X,\gamma)$. If $g \in \Mon(\CF_M)$ and $q_M(g)$ is an identity element, there must exist an $\CF_M$-path $(X,\gamma)$ which represents $g$ and a transversal $\tau$ to the leaves of $\CF_M$ such that
\[ \Phi^1_X|_{\tau} = \text{Id}_\tau. \]
Since a submanifold transverse to $\CF_M$ must contain a submanifold transverse to $\CF_N$, it follows that the holonomy of $(X,\gamma)$ relative to $\CF_N$ is also trivial.

For the other case we take $\CF_M = f^{-1}(\CF_M)$. Given a transversal $\tau$ to $\CF_M$, then $f(\tau)$ is a transversal to $\CF_N$. Both $\tau$ and $f(\tau)$ inherit foliations from $\CF_M$ and $\CF_N$ respectively. From the definition of the pullback foliation, it follows that:
\[ f|_{\tau} \colon \tau \to f(\tau) \]
is a foliation preserving diffeomorphism.

Recall the discussion of projectable $\CF_M$-paths from Example~\ref{ex:ME}.
Suppose $g \in \Mon(\CF_M)$ is represented by a projectable $\CF_M$-path $(X,\gamma)$ and $q_M(g)$ is an identity. Given any transversal $\tau$ through $\gamma(0)$ we have that
\[ f \circ \Phi^1_X|_{\tau} = \Phi^1_{\d f (X)}|_{f(\tau)} \]
Since $f|_{\tau} \colon \tau \to f(\tau)$ is a foliation preserving diffeomorphism, it follows that the holonomy of $(X,\gamma)$ is trivial if and only if the holonomy of $(\d f (X), f \circ \gamma)$ is trivial.
\end{proof}
\begin{rem}
According to Example~\ref{ex:factorization}, the condition in Theorem~\ref{thm:holfunc} that $(F,f)$ is a morphism of foliated manifolds is equivalent to the claim that $\CF_M$ is contained in $f^{-1}(\CF_N)$. This result is an improvement on theorems that appear in preprints \cite{SingSub} and \cite{lie2grps}. However, by decomposing the problem into the cases $f = \text{Id}_M$ and $\CF_M = f^{-1}(\CF_N)$, one can mostly recover Theorem~\ref{thm:holfunc} as a corollary of these papers.
\end{rem}
\begin{ex}\label{ex:ME1} Let $\pi: P\fto (M,\CF_M)$ be a surjective submersion with connected fibers. On $P$, take the foliation $\CF_P:=\pi^{-1}(\CF_M)$. In a recent preprint~\cite{ME2018}, it is shown that:
\[ \CH(\CF_P)\cong \pi^{-1}\CH(\CF_M) := P \times_{\pi,t} \CH(\CF_M) \times_{s, \pi} P \]
The canonical projection map $\pi^{-1}\CH(\CF_M)\fto \CH(\CF_M)$ is $\Hol(\d\pi_*, \pi)$ from Theorem~\ref{thm:holfunc}.

If we are given a Hausdorff Morita equivalence $N \leftarrow P \to M$ as defined in \cite{ME2018}, we get a pair of morphisms foliated manifolds $\CF_M \leftarrow \CF_P \to \CF_N$ and a commuting diagram:
\[
\begin{tikzcd}
\Mon(\CF_N) \arrow{d} & \Mon(\CF_P) \arrow{r}\arrow{d} \arrow{l} & \Mon(\CF_M) \arrow{d} \\
\CH(\CF_N) & \CH(\CF_P) \arrow{r} \arrow{l} & \CH(\CF_M) \\
\end{tikzcd}
\]
On the bottom row, the diffeological groupoid morphisms are weak equivalences (per \cite{ME2018}). On the top row, the horizontal arrows are diffeological groupoid fibrations which are not necessarily weak equivalences.

If one wishes to obtain a weak equivalence at the top level, then one needs to put additional conditions on the submersions $N \leftarrow P \to M$. To see which conditions are needed, we refer the reader to Theorem 1.2 in \cite{Villatoro2018}.
\end{ex}

\bibliographystyle{plain}
\bibliography{MEfolbib}

\end{document}